\documentclass[a4paper,10pt]{amsart}
\usepackage{amsmath}
\usepackage{amssymb,color}

\usepackage[colorlinks]{hyperref}
\definecolor{linkblue}{RGB}{1,1,190}
\definecolor{citered}{RGB}{190,1,1}
\hypersetup{linkcolor=linkblue,urlcolor=linkblue,citecolor=citered}

\theoremstyle{plain}
\newtheorem{theorem}{\bf Theorem}[section]
\newtheorem{proposition}[theorem]{\bf Proposition}
\newtheorem{lemma}[theorem]{\bf Lemma}
\newtheorem{corollary}[theorem]{\bf Corollary}

\theoremstyle{definition}

\newtheorem{remark}[theorem]{\bf Remark}

\newcommand{\N}{\mathbb N}
\newcommand{\Z}{\mathbb Z}

\DeclareMathOperator{\ord}{ord}

\DeclareMathOperator{\spec}{spec}\DeclareMathOperator{\supp}{supp}

\DeclareMathOperator{\End}{End}

\newcommand{\DP}{\negthinspace :\negthinspace}
\newcommand{\red}{{\text{\rm red}}}

\newcommand{\id}{{\text{\rm id}}}
\renewcommand{\t}{\,|\,}

\numberwithin{equation}{section}

\thanks{This work was supported by the Austrian Science Fund FWF (Project Number P36852-N) and by the National Natural Science Foundation of China (Grant No. 12001331)}

\begin{document}

\title[On Monoids of plus-minus weighted Zero-Sum Sequences]{On Monoids of plus-minus weighted Zero-Sum Sequences:\\ The Isomorphism Problem and the Characterization Problem}

\author{Florin Fabsits and Alfred Geroldinger and Andreas Reinhart and Qinghai Zhong}

\address{University of Graz, NAWI Graz\\
Department of Mathematics and Scientific Computing\\
Heinrichstra{\ss}e 36\\
8010 Graz, Austria (F.~Fabsits, A.~Geroldinger, A.~Reinhart, and Q.~Zhong)}

\address{School of Mathematics and Statistics, Shandong University of Technology, Zibo, Shandong 255000, China (Q.~Zhong)}
\email{alfred.geroldinger@uni-graz.at, andreas.reinhart@uni-graz.at, qinghai.zhong@uni-graz.at}
\urladdr{https://imsc.uni-graz.at/geroldinger, https://imsc.uni-graz.at/reinhart/, https://imsc.uni-graz.at/zhong/}

\begin{abstract}
Let $G$ be an additive abelian group. A sequence $S=g_1\cdot\ldots\cdot g_{\ell}$ of terms from $G$ is a plus-minus weighted zero-sum sequence if there are $\varepsilon_1,\ldots,\varepsilon_{\ell}\in\{1,-1\}$ such that $\varepsilon_1 g_1+\ldots+\varepsilon_{\ell} g_{\ell}=0$. We first characterize (in terms of $G$) when the monoid $\mathcal{B}_{\pm}(G)$ of plus-minus weighted zero-sum sequences is Mori resp. Krull resp. finitely generated. After that we study the Isomorphism and the Characterization Problem for monoids of plus-minus weighted zero-sum sequences.
\end{abstract}

\maketitle

\smallskip
\section{Introduction}\label{1}
\smallskip

Let $G$ be an additively written abelian group. We consider (finite unordered) sequences (with repetition allowed) of terms from $G$ as elements of the (multiplicatively written) free abelian monoid $\mathcal{F}(G)$ with basis $G$. Let $\Gamma\subset\End(G)$ be a non-empty subset of the endomorphism group of $G$. A sequence $S= g_1\cdot\ldots\cdot g_{\ell}\in\mathcal{F}(G)$ is called a ($\Gamma$-)weighted zero-sum sequence if there are $\gamma_1,\ldots,\gamma_{\ell}\in\Gamma$ such that $\gamma_1(g_1)+\ldots+\gamma_{\ell}(g_{\ell})=0$. Then the set $\mathcal{B}_{\Gamma}(G)$ of all $\Gamma$-weighted zero-sum sequences over $G$ is a submonoid of $\mathcal{F}(G)$. A special emphasis has been laid on the case $\Gamma=\{\id_G,-\id_G\}$ and in that case one speaks of plus-minus weighted zero-sum sequences and the associated monoid is denoted by $\mathcal{B}_{\pm}(G)$.

Since the last decade combinatorial and number theoretic problems of weighted zero-sum sequences have seen a lot of interest. Many of the classical zero-sum invariants (including the Davenport constant $\mathsf{D}(G)$, the Erd{\H{o}}s-Ginzburg-Ziv constant $\mathsf{s}(G)$, Gao's constant $\mathsf{E}(G)$, the Harborth constant $\mathsf{g}(G)$, and more) have found their weighted analogs (for a weighted version of $\mathsf{E}(G)$ see \cite[Chapter 16]{Gr13a}, for connections with coding theory see \cite{Ma-Or-Sa-Sc15}, and for a sample of papers with a strong number theoretic flavor see \cite{Ad-Gr-Su12a,Gr06a,Gr-He15a,Gr-Ma-Or12,Gr-Ph-Po13,HK14b,Ma-Or-Ra-Sc13a,Ma-Or-Ra-Sc16a,Ma-Or-Sc14a}; see also the remark at the end of Section \ref{2}).

Algebraic properties of the monoid of weighted zero-sum sequences were first studied by Schmid et al. in \cite{B-M-O-S22}. There are transfer homomorphisms from norm monoids in orders of algebraic number fields (and others) to monoids of weighted zero-sum sequences (see \cite[Theorem 7.1]{B-M-O-S22}, \cite[Theorems 3.2 and 3.5]{Ge-HK-Zh22}). This implies that arithmetic questions in norm monoids (in particular, their sets of lengths) can be studied in monoids of weighted zero-sum sequences, and it demonstrates the connection of the latter with questions in commutative ring theory.

In the present paper we focus on the monoid $\mathcal{B}_{\pm}(G)$ of plus-minus weighted zero-sum sequences. In Section~\ref{3}, we characterize when $\mathcal{B}_{\pm}(G)$ is a Mori monoid (Theorem~\ref{3.4}), when it is a Krull monoid (Corollary~\ref{3.5}) and when it is finitely generated resp. a C-monoid (Theorem~\ref{3.7}).

In Section~\ref{4}, we study the Isomorphism and the Characterization Problem. We recall these two problems.

\smallskip
\noindent
{\bf The Isomorphism Problem} (for plus-minus weighted zero-sum sequences).
Let $G_1$ and $G_2$ be abelian groups such that the monoids $\mathcal{B}_{\pm}(G_1)$ and $\mathcal{B}_{\pm}(G_2)$ are isomorphic. Are the groups $G_1$ and $G_2$ isomorphic?

\smallskip
It is well-known that the Isomorphism Problem has an affirmative answer for monoids of (unweighted) zero-sum sequences over abelian groups (\cite[Corollary 2.5.7]{Ge-HK06a}). The Isomorphism Problem was recently studied for power monoids of numerical monoids (\cite{Bi-Ge23,Tr-Ya24a}) and for monoids of product-one sequences over non-abelian groups (\cite{Ge-Oh24a}). For the Isomorphism Problem for group rings we refer to the survey \cite{Ma22}. In the present paper we give an affirmative answer to the Isomorphism Problem for plus-minus weighted zero-sum sequences in case that one group is a direct sum of cyclic groups (Theorem~\ref{4.3}).

\smallskip
The Characterization Problem asks whether or not two monoids (or domains), from a given class of monoids, are already uniquely determined by their systems of sets of lengths. This problem has its origin in algebraic number theory. Indeed, in the 1970s Narkiewicz asked wether or not the ideal class group of a number field can be characterized by arithmetic properties of the ring of integers. Nowadays, this question got reformulated in terms of monoids of (unweighted) zero-sum sequences over finite abelian groups, where an affirmative answer is expected (for an overview, we refer to the survey \cite{Ge-Zh20a}).

\smallskip
\noindent
{\bf The Characterization Problem} (for plus-minus weighted zero-sum sequences).
Let $G_1$ and $G_2$ be finite abelian groups with Davenport constant $\mathsf{D}_{\pm}(G_1)\ge 4$ such that their systems of sets of lengths $\mathcal{L}\big(\mathcal{B}_{\pm}(G_1)\big)$ and $\mathcal{L}\big(\mathcal{B}_{\pm}(G_2)\big)$ coincide. Are the groups $G_1$ and $G_2$ isomorphic?

\smallskip
Clearly, a necessary condition for an affirmative answer to the Characterization Problem (for a class of abelian groups) is an affirmative answer to the Isomorphism Problem.
Any work on the Characterization Problem (both for unweighted as well as for weighted zero-sum sequences) requires a lot of ingredients from additive combinatorics. If $\mathcal{L}\big(\mathcal{B}_{\pm}(G_1)\big)=\mathcal{L}\big(\mathcal{B}_{\pm}(G_2)\big)$ resp. $\mathcal{L}\big(\mathcal{B}(G_1)\big)=\mathcal{L}\big(\mathcal{B}(G_2)\big)$, then one easily gets that for the associated Davenport constants we have $\mathsf{D}_{\pm}(G_1)=\mathsf{D}_{\pm}(G_2)$ resp. $\mathsf{D}(G_1)=\mathsf{D}(G_2)$. In spite of being studied since decades, the precise value of the Davenport constant of a finite abelian group $G$ (in terms of the group invariants) is unknown for general groups of rank $\mathsf{r}(G)\ge 3$. In Section~\ref{4}, we settle the Characterization Problem for plus-minus weighted zero-sum sequences in the case when $G_1$ is a cyclic group of odd order (Theorem~\ref{4.6}).

\smallskip
\section{Prerequisites}\label{2}
\smallskip

We denote by $\mathbb P\subset\N\subset\N_0\subset\Z$ the sets of prime number, positive integers, non-negative integers, and integers. For $a,b\in\Z$, let $[a,b ]=\{x\in\Z\colon a\le x\le b\}$ be the discrete interval between $a$ and $b$. For subsets $A,B\subset\Z$, we denote by $A+B=\{a+b\colon a\in A,b\in B\}$ the {\it sumset} of $A$ and $B$, and for $k\in\Z$, we set $k+A=\{k\}+A$. For $k\in\N$, let $\N_{\geq k}=\{a\in\mathbb{N}\colon a\geq k\}$ and let $k\cdot A=\{ka\colon a\in A\}$ be the {\it dilation} of $A$ by $k$. If $A=\{m_0,\ldots,m_k\}\subset\Z$, with $k\in\N_0$ and $m_0<\ldots<m_k$, then $\Delta(A)=\{m_i-m_{i-1}\colon i\in [1,k]\}$ is the {\it set of distances} of $A$.

By a {\it monoid}, we mean a commutative cancellative semigroup with identity element, and we use multiplicative notation. Let $H$ be a monoid. Then $H^{\times}$ denotes the group of invertible elements and $\mathsf{q}(H)$ denotes the quotient group. Furthermore, let
\begin{itemize}
\item $H'=\{x\in\mathsf{q}(H)\colon\text{there is some }n\in\N\text{ such that }x^m\in H\text{ for each }m\in\N_{\geq n}\}$ be the \,{\it seminormal closure}\, of $H$,
\item $\widetilde{H}=\{x\in\mathsf{q}(H)\colon\text{there is some }n\in\N\text{ such that }x^n\in H\}$ be the \,{\it root closure} of $H$, and let
\item $\widehat{H}=\{x\in\mathsf{q}(H)\colon\text{there is some }c\in H\text{ such that }cx^n\in H\ \text{for all}\ n\in\N\}$ be the \,{\it complete integral closure}\, of $H$.
\end{itemize}
Then $H\subset H'\subset\widetilde H\subset\widehat H\subset\mathsf{q}(H)$, and $H$ is said to be {\it seminormal}, or {\it root closed}, or {\it completely integrally closed} if $H=H^{\prime}$, or $H=\widetilde H$, or $H=\widehat H$. For a set $P$, we denote by $\mathcal{F}(P)$ the free abelian monoid with basis $P$, and we will use multiplicative notation for $\mathcal{F}(P)$. The monoid $H$ is {\it factorial} if its associated reduced monoid $H_{\red}=\{aH^{\times}\colon a\in H\}$ is free abelian.
A monoid homomorphism $\varphi\colon H\to D$ is said to be a
\begin{itemize}
\item {\it divisor homomorphism} if $a,b\in H$ and $\varphi(a)\t\varphi(b)$ (in $D$) implies that $a\t b$ (in $H$),
\item {\it divisor theory} (for $H$) if $D$ is free abelian, $\varphi$ is a divisor homomorphism, and for every $\alpha\in D$ there are $a_1,\ldots,a_m\in H$ such that $\alpha=\gcd\big(\varphi(a_1),\ldots,\varphi(a_m)\big)$.
\end{itemize}
If $\varphi\colon H\to D$ is a divisor theory, then $\mathcal{C}(H)=\mathsf{q}(D)/\mathsf{q}(\varphi(H))$ is the {\it $($divisor$)$ class group} of $H$. If $H$ is a submonoid of $D$, then it is easily checked that the inclusion $H\hookrightarrow D$ is a divisor homomorphism if and only if $H=\mathsf{q}(H)\cap D$.

\smallskip
\noindent
{\bf Ideal Theory of Monoids.} Our main references are \cite{Ge-HK06a,HK98}. To fix notation, we gather some key concepts needed in the sequel. For subsets $I,J\subset\mathsf{q}(H)$, we set $(I\DP J)=\{x\in\mathsf{q}(H)\colon xJ\subset I\}$, $I^{-1}=(H\DP I)$, and $I_v=(I^{-1})^{-1}$. Then $I\subset H$ is called an $s$-{\it ideal} if $IH=I$ and it is called a {\it divisorial ideal} (or a {\it $v$-ideal}) if $I=I_v$. We denote by $s$-$\spec(H)$ the set of prime $s$-ideals of $H$ and by $\mathfrak{X}(H)$ the set of minimal non-empty prime $s$-ideals of $H$.
The monoid $H$ is said to be a
\begin{itemize}
\item {\it Mori monoid} if it satisfies the ascending chain condition on divisorial ideals,
\item {\it Krull monoid} if it is a completely integrally closed Mori monoid (equivalently, if it has a divisor theory).
\end{itemize}
If $H$ is a Krull monoid, then every $v$-ideal is $v$-invertible and the monoid of $v$-ideals is free abelian (with $v$-multiplication as operation) with basis $\mathfrak{X}(H)$. If $\mathcal{F}_v(H)$ denotes the semigroup of fractional $v$-ideals, then $\mathcal{C}_v(H)=\mathcal{F}_v(H)^{\times}/\{aH\colon a\in\mathsf{q}(H)\}$ is the {\it $v$-class group} of $H$. If $H$ is a Krull monoid, then $\mathcal{C}_v(H)$ is isomorphic to the divisor class group of $H$.

We need the concept of C-monoids (for details see \cite[Chapter 2]{Ge-HK06a}). Let $F$ be a factorial monoid and let $H\subset F$ be a submonoid. Two elements $y,y'\in F$ are called $H$-equivalent if $y^{-1}H\cap F={y'}^{-1}H\cap F$, equivalently, if
\[
\text{for all}\ x\in F,\ \text{we have}\ xy\in H\ \text{if and only if}\ x y'\in H\,.
\]
This defines a congruence relation on $F$, and for $y\in F$, we denote by $|y]_H^F=[y]$ its congruence class. Then
\[
\mathcal{C}^*(H,F)=\big\{[y]\colon y\in (F\setminus F^{\times})\cup\{1\}\big\}\subset\mathcal{C}(H,F)=\big\{[y]\colon y\in F\big\}
\]
are commutative semigroups with identity element $[1]$, $\mathcal{C}(H,F)$ is the {\it class semigroup}, and $\mathcal{C}^*(H,F)$ is the {\it reduced class semigroup} of $H$ in $F$. A monoid $H$ is said to be a C-{\it monoid} (defined in $F$) if it is a submonoid of $F$ such that $H\cap F^{\times}=H^{\times}$ and $\mathcal{C}^*(H,F)$ is finite. If $H$ is a C-monoid, then $H$ is Mori, $(H\DP\widehat H)\ne\emptyset$, and $\mathcal{C}(\widehat H)$ is finite. Every Krull monoid with finite class group is a C-monoid and in that case the class semigroup and the class group coincide. A commutative ring is a C-ring if its monoid of regular elements is a C-monoid (for a sample of work on C-monoids and C-rings, see \cite{Br20a,Br21a,Fa-Zh22a,Ka16b,Oh20a,Re13a}).

\smallskip
\noindent
{\bf Arithmetic Theory of Monoids.} We denote by $\mathcal{A}(H)$ the set of atoms of $H$ and we say that $H$ is {\it atomic} if every non-invertible element of $H$ can be written as a finite product of atoms. If $a=u_1\cdot\ldots\cdot u_k\in H$, where $k\in\N$ and $u_1,\ldots,u_k\in\mathcal{A}(H)$, then $k$ is called a factorization length of $a$ and the set $\mathsf{L}(a)\subset\N$ of all possible factorization lengths of $a$ is called the {\it set of lengths} of $a$. It is convenient to set $\mathsf{L}(a)=\{0\}$ for $a\in H^{\times}$ and then
$\mathcal{L}(H)=\{\mathsf{L}(a)\colon a\in H\}$ denotes the {\it system of sets of lengths} of $H$. Thus, $H$ is atomic if and only if all sets of lengths are non-empty. Furthermore, $H$ is said to be
\begin{itemize}
\item {\it half-factorial} if $|L|=1$ for all $L\in\mathcal{L}(H)$,
\item a BF-{\it monoid} if all $L\in\mathcal{L}(H)$ are finite and non-empty.
\end{itemize}
Every Mori monoid is a BF-monoid and every factorial monoid is half-factorial.

\smallskip
\noindent
{\bf Sequences over abelian groups.} Let $G$ be an additively written abelian group and let $G_0\subset G$ be a subset. Then $[G_0]\subset G$ denotes the submonoid generated by $G_0$ and $\langle G_0\rangle\subset G$ is the subgroup generated by $G_0$. Let ${\rm exp}(G)\in\mathbb{N}\cup\{\infty\}$ denote the exponent of $G$. For $n\in\N$, let $C_n$ be an additive cyclic group with $n$ elements and let $nG=\{ng\colon g\in G\}$. For a prime $p\in\mathbb P$, $G$ is called an elementary $p$-group if $pG=0$ (equivalently, every nonzero element has order $p$). We denote by $\mathcal{F}(G)$ the multiplicatively written free abelian monoid with basis $G$. In additive combinatorics, elements of $\mathcal{F}(G)$ are called {\it sequences} over $G$. Let
\[
S=g_1\cdot\ldots\cdot g_{\ell}=\prod_{g\in G} g^{\mathsf{v}_g(S)}
\]
be a sequence over $G$, where $\ell\in\N_0$, $g_1,\ldots,g_{\ell}\in G$, and $\mathsf{v}_g(S)\in\N_0$ with $\mathsf{v}_g(S)=0$ for almost all $g\in G$. Then
\begin{itemize}
\item $|S|=\ell\in\N_0$ is the {\it length} of $S$,
\item $\supp(S)=\{g\in G\colon\mathsf{v}_g(S)>0\}\subset G$ is the {\it support} of $S$,
\item $\mathsf{h}(S)=\max\{\mathsf{v}_g(S)\colon g\in G\}$ is the {\it maximum multiplicity} of a term of $S$,
\item $\sigma(S)=g_1+\ldots+g_{\ell}=\sum_{g\in G}\mathsf{v}_g(S)g\in G$ is the {\it sum} of $S$, and
\item $\sigma_{\pm}(S)=\{\varepsilon_1 g_1+\ldots+\varepsilon_{\ell} g_{\ell}\colon\varepsilon_1,\ldots ,\varepsilon_{\ell}\in\{-1,1\}\}$ is the {\it set of plus-minus weighted sums} of $S$.
\end{itemize}
The sequence $S$ is called a
\begin{itemize}
\item {\it zero-sum sequence} if $\sigma(S)=0$,
\item {\it plus-minus weighted zero-sum sequence} if $0\in\sigma_{\pm}(S)$.
\end{itemize}
Then
\[
\mathcal{B}(G_0)=\{S\in\mathcal{F}(G_0)\colon S\ \text{is a zero-sum sequence}\}\subset\mathcal{F}(G_0)
\]
is the {\it monoid of zero-sum sequences} over $G_0$,
\[
\mathcal{B}_{\pm}(G_0)=\{S\in\mathcal{F}(G_0)\colon S\ \text{is a plus-minus weighted zero-sum sequence}\}
\]
is the {\it monoid of plus-minus weighted zero-sum sequences} over $G_0$, and we have
\[
\mathcal{B}(G_0)\subset\mathcal{B}_{\pm}(G_0)\subset\mathcal{F}(G_0)\,.
\]
Furthermore,
\begin{itemize}
\item $\mathsf{D}(G_0)=\sup\big\{|S|\colon S\in\mathcal{A}\big(\mathcal{B}(G_0)\big)\big\}$ is the {\it Davenport constant} of $G_0$, and
\item $\mathsf{D}_{\pm}(G_0)=\sup\big\{|S|\colon S\in\mathcal{A}\big(\mathcal{B}_{\pm}(G_0)\big)\big\}$ is the {\it plus-minus weighted Davenport constant} of $G_0$.
\end{itemize}
It is easy to see that $\mathcal{B}(G_0)$ and $\mathcal{B}_{\pm}(G_0)$ are BF-monoids, and
it is well-known that each of $\mathcal{A}\big(\mathcal{B}(G)\big)$, $\mathcal{A}\big(\mathcal{B}_{\pm}(G)\big)$, $\mathsf{D}(G)$, and $\mathsf{D}_{\pm}(G)$ is finite if and only if $G$ is finite.

We end with a remark on notation. For sequences over $G_0$ as well as for plus-minus weighted sequences over $G_0$ the following three properties have been studied. For simplicity, we formulate them only for plus-minus weighted sequences.
\begin{itemize}
\item[(a)] What is the smallest integer $N\in\N$ such that every sequence $S\in\mathcal{F}(G_0)$ has a plus-minus weighted zero-sum subsequence?

\item[(b)] What is the maximal length of a sequence that has no plus-minus weighted zero-sum subsequence?

\item[(c)] What is the maximal length of a minimal plus-minus weighted zero-sum sequence?
\end{itemize}
All these integers have been called ``{\it weighted Davenport constant}\,'' and were denoted as $\mathsf{D}_{\pm}(G_0)$ or as $\mathsf{d}_{\pm}(G_0)$ or similarly. The constant addressed in (c) fits into the general concept of a ``Davenport constant of a monoid, embedded in a free abelian monoid'', as introduced in \cite{Cz-Do-Ge16a} and further used in \cite{B-M-O-S22}. Since in the present paper, we do not need constants fulfilling properties (a) and (b), we use the shorthand notation $\mathsf{D}_{\pm}(G_0)$ for the Davenport constant of the monoid $\mathcal{B}_{\pm}(G_0)$, whence we have
\[
\mathsf{D}_{\pm}(G_0)=\mathsf{D}\big(\mathcal{B}_{\pm}(G_0)\big)\,,
\]
where the latter notation is used in \cite{B-M-O-S22,Cz-Do-Ge16a}.

\smallskip
\section{Characterizations of Ideal Theoretic Properties}\label{3}
\smallskip

In this section we study algebraic properties of the monoid $\mathcal{B}_{\pm}(G)$. We characterize when it is Mori or Krull (equivalently, completely integrally closed resp. root closed) or finitely generated (equivalently, a C-monoid) (Theorems~\ref{3.4},~\ref{3.7}, and Corollary~\ref{3.5}).

\smallskip
\begin{lemma}\label{3.1}
Let $G$ be an abelian group.
\begin{enumerate}
\item If $|G|\le 2$, then
\[
\mathcal{B}(G)=\mathcal{B}_{\pm}(G)\cong\mathcal{F}(G)\cong (\N_0^{|G|},+)\,.
\]
\item The following statements are equivalent.
\begin{enumerate}
\item $|G|\le 2$.
\item $\mathcal{B}_{\pm}(G)$ is factorial.
\item $\mathcal{B}_{\pm}(G)$ is half-factorial.
\end{enumerate}
\end{enumerate}
\end{lemma}

\begin{proof}
1. Obvious.

2. (a) $\Longrightarrow$ (b) This follows from 1.

(b) $\Longrightarrow$ (c) Obvious.

(c) $\Longrightarrow$ (a) We suppose that $|G|\ge 3$ and show that $\mathcal{B}_{\pm}(G)$ is not half-factorial. To do so, it is sufficient to find some atom $U=g_1\cdot\ldots\cdot g_{\ell}\in\mathcal{B}_{\pm}(G)$ with $|U|\ge 3$. Then we have $U^2=(g_1^2)\cdot\ldots\cdot (g_{\ell}^2)$.

If there is an element $g\in G$ with $\ord(g)=n\ge 3$ odd, then $U=g^n$ is an atom. If there is an element $g\in G$ with $\ord(g)=\infty$, then $U=g^2(2g)$ is an atom. If there are two distinct elements $e_1,e_2$ of order two, then $U=e_1e_2(e_1+e_2)$ is an atom. If none of these conditions hold, then $G$ has an element $g$ with $\ord(g)=4$, whence $U=g^2(2g)$ is an atom.
\end{proof}

\smallskip
\begin{theorem}\label{3.2}
Let $G$ be an abelian group.
\begin{enumerate}
\item $\widetilde{\mathcal{B}_{\pm}(G)}=\widehat{\mathcal{B}_{\pm}(G)}$ is a Krull monoid.
\item If $|G|\ne 2$, then the inclusion $\widetilde{\mathcal{B}_{\pm}(G)}\hookrightarrow\mathcal{F}(G)$ is a divisor theory. Its class group is isomorphic to a factor group of $G$ and every class contains a prime divisor.
\end{enumerate}
\end{theorem}

\begin{proof}
If $|G|\le 2$, then all statements hold true by Lemma~\ref{3.1}. Thus, we suppose that $|G|\ge 3$.
We set
\[
\mathcal{B}_{\pm}^*(G)=\mathsf{q}\big(\mathcal{B}_{\pm}(G)\big)\cap\mathcal{F}(G)\,.
\]
Since
\[
\mathcal{B}_{\pm}(G)\subset\mathcal{B}_{\pm}^*(G)\quad\text{and}\quad\mathcal{B}_{\pm}^*(G)\subset\mathsf{q}\big(\mathcal{B}_{\pm}(G)\big)\,,
\]
it follows that
\[
\mathsf{q}\big(\mathcal{B}_{\pm}(G)\big)\subset\mathsf{q}\big(\mathcal{B}_{\pm}^*(G)\big)\subset\mathsf{q}\big(\mathcal{B}_{\pm}(G)\big)\,,
\]
This implies that
\[
\mathcal{B}_{\pm}^*(G)=\mathsf{q}\big(\mathcal{B}_{\pm}^*(G)\big)\cap\mathcal{F}(G)\,,
\]
whence $\mathcal{B}_{\pm}^*(G)\hookrightarrow\mathcal{F}(G)$ is a divisor homomorphism and $\mathcal{B}_{\pm}^*(G)$ is a Krull monoid. Thus, $\mathcal{B}_{\pm}^*(G)$ is completely integrally closed, which implies that
\[
\widetilde{\mathcal{B}_{\pm}(G)}\subset\widehat{\mathcal{B}_{\pm}(G)}\subset\mathcal{B}_{\pm}^*(G)\,.
\]
If $S\in\mathsf{q}\big(\mathcal{B}_{\pm}(G)\big)\cap\mathcal{F}(G)$, then $S^2\in\mathcal{B}_{\pm}(G)$, whence
\[
\mathsf{q}\big(\mathcal{B}_{\pm}(G)\big)\cap\mathcal{F}(G)=\mathcal{B}_{\pm}^*(G)\subset\widetilde{\mathcal{B}_{\pm}(G)}
\]
and thus it follows that
\[
\widetilde{\mathcal{B}_{\pm}(G)}=\widehat{\mathcal{B}_{\pm}(G)}=\mathcal{B}_{\pm}^*(G)\,.
\]
By \cite[Proposition 2.5.6]{Ge-HK06a}, the inclusion $\mathcal{B}(G)\hookrightarrow\mathcal{F}(G)$ is a divisor theory with class group
\[
\mathsf{q}\big(\mathcal{F}(G)\big)\big/\mathsf{q}\big(\mathcal{B}(G)\big)\cong G
\]
and every class contains precisely one prime divisor.
Thus, every $S\in\mathcal{F}(G)$ is a greatest common divisor of elements from $\mathcal{B}(G)$ and hence it is a greatest common divisor of elements from $\widetilde{\mathcal{B}_{\pm}(G)}$. Therefore the inclusion $\widetilde{\mathcal{B}_{\pm}(G)}\hookrightarrow\mathcal{F}(G)$ is a divisor theory with class group
\[
\mathsf{q}\big(\mathcal{F}(G)\big)\big/\mathsf{q}\big(\mathcal{B}_{\pm}(G)\big)\cong\mathsf{q}\big(\mathcal{F}(G)\big)\big/\mathsf{q}\big(\mathcal{B}(G)\big)\Big/\mathsf{q}\big(\mathcal{B}_{\pm}(G)\big)\big/\mathsf{q}\big(\mathcal{B}(G)\big)\cong G\Big/\mathsf{q}\big(\mathcal{B}_{\pm}(G))\big/\mathsf{q}\big(\mathcal{B}(G)\big)\,.\qedhere
\]
\end{proof}

\smallskip
\begin{lemma}\label{3.3}
Let $G$ be an abelian group and let $g\in G$ with $\ord(g)=\infty$. For every $n\in\mathbb{N}$, let $S_n=g\left(2^{n+1}g\right)^2\left(\left(2^{n+2}-1\right)g\right)$ and $a_n=(2g)\prod_{j=1}^n\left((3\cdot 2^j)g\right)$.
\begin{enumerate}
\item For every $n\in\mathbb{N}$, $S_n\in\mathcal{B}_{\pm}(G)$, $a_n\in\mathsf{q}\big(\mathcal{B}_{\pm}(G)\big)$, $a_nS_{n+1}\not\in\mathcal{B}_{\pm}(G)$, and $a_n S_i\in\mathcal{B}_{\pm}(G)$ for each $i\in [1,n]$.
\item $\mathcal{B}_{\pm}(G)$ is not a Mori monoid.
\end{enumerate}
\end{lemma}

\begin{proof}
Let $Q$ be the quotient group of $\mathcal{B}_{\pm}(G)$.

1. Let $n\in\mathbb{N}$. Since $1-2^{n+1}-2^{n+1}+(2^{n+2}-1)=0$, we have that $g+(-1)(2^{n+1}g)+(-1)(2^{n+1}g)+(2^{n+2}-1)g=0$, and thus $S_n\in\mathcal{B}_{\pm}(G)$. Next we show that for each $z\in 2\mathbb{Z}$, $zg\in Q$. Let $z\in 2\mathbb{Z}$. Then $(\frac{z}{2}g)^2,(zg)(\frac{z}{2}g)^2\in\mathcal{B}_{\pm}(G)$, and hence $zg=\frac{(zg)(\frac{z}{2}g)^2}{(\frac{z}{2}g)^2}\in Q$. Consequently, $a_n\in Q$.

\smallskip
Now we prove that $a_nS_{n+1}\not\in\mathcal{B}_{\pm}(G)$. Assume that $a_nS_{n+1}\in\mathcal{B}_{\pm}(G)$. Then there are some $\varepsilon,(\varepsilon_j)_{j=1}^n$ and $(\eta_i)_{i=1}^4$ such that $\varepsilon,\varepsilon_j,\eta_i\in\{-1,1\}$ for each $j\in [1,n]$ and $i\in [1,4]$ and $\varepsilon(2g)+\sum_{j=1}^n\varepsilon_j((3\cdot 2^j)g)+\eta_1g+(\eta_2+\eta_3)(2^{n+2}g)+\eta_4((2^{n+3}-1)g)=0$. Since $\ord(g)=\infty$, it follows that $2\varepsilon+\sum_{j=1}^n 3\varepsilon_j2^j+\eta_1+(\eta_2+\eta_3)2^{n+2}+\eta_4(2^{n+3}-1)=0$. Without restriction, we can assume that $\eta_4=1$. Suppose that $\eta_3=1$. Then $2^{n+2}+2^{n+3}-1=|2\varepsilon+\sum_{j=1}^n 3\varepsilon_j2^j+\eta_1+\eta_2 2^{n+2}|\leq 2+3\sum_{j=1}^n 2^j+1+2^{n+2}=2^{n+2}+3\sum_{j=0}^n 2^j=2^{n+2}+3(2^{n+1}-1)=2^{n+2}+3\cdot 2^{n+1}-3<2^{n+2}+2^{n+3}-1$, a contradiction. Consequently, $\eta_3=-1$. It follows by analogy that $\eta_2=-1$. Therefore $2\varepsilon+\sum_{j=1}^n 3\varepsilon_j2^j+\eta_1-1=0$, and hence $3\cdot 2^n=|2\varepsilon+\sum_{j=1}^{n-1} 3\varepsilon_j2^j+\eta_1-1|\leq 2+\sum_{j=1}^{n-1} 3\cdot 2^j+1+1=1+3\sum_{j=0}^{n-1} 2^j=1+3(2^n-1)=3\cdot 2^n-2<3\cdot 2^n$, a contradiction.

\smallskip
Let $i\in [1,n]$. Note that $2(-1)^{i+1}+\sum_{j=1}^{i-1} 3(-1)^{i+1-j}2^j-\sum_{j=i}^{n-1} 3\cdot 2^j+3\cdot 2^n-1+2^{i+1}-2^{i+1}-(2^{i+2}-1)=2(-1)^{i+1}+3(-1)^{i+1}\sum_{j=1}^{i-1} (-2)^j-3\sum_{j=i}^{n-1} 2^j+3\cdot 2^n-2^{i+2}=(-1)^{i+1}(2+3(\frac{(-2)^i-1}{-2-1}-1))-3(\frac{2^n-1}{2-1}-\frac{2^i-1}{2-1})+3\cdot 2^n-2^{i+2}=(-1)^{i+1}(2-((-2)^i+2))-3(2^n-2^i)+3\cdot 2^n-2^{i+2}=(-1)^{i+2}(-2)^i-2^i=(-1)^{2i+2}2^i-2^i=0$. We infer that $(-1)^{i+1}(2g)+\sum_{j=1}^{i-1} (-1)^{i+1-j}((3\cdot 2^j)g)+\sum_{j=i}^{n-1} (-1)((3\cdot 2^j)g)+(3\cdot 2^n)g+(-1)g+2^{i+1}g+(-1)(2^{i+1}g)+(-1)((2^{i+2}-1)g)=0$. Therefore $a_nS_i\in\mathcal{B}_{\pm}(G)$.

\bigskip
2. It follows from 1. that for each $n\in\mathbb{N}$, $(\mathcal{B}_{\pm}(G)\colon\{S_i\colon i\in [1,n]\})\supsetneq (\mathcal{B}_{\pm}(G)\colon\{S_i\colon i\in [1,n+1]\})$, and hence $(\{S_i\colon i\in [1,n]\})_v\subsetneq (\{S_i\colon i\in [1,n+1]\})_v$ for each $n\in\mathbb{N}$. Therefore $\mathcal{B}_{\pm}(G)$ is not a Mori monoid.
\end{proof}

Let $H$ and $D$ be monoids and let $\varphi\colon H\rightarrow D$ be a divisor homomorphism. Note that $H$ is seminormal if and only if for each $x\in\mathsf{q}(H)$ with $x^2,x^3\in H$, we have that $x\in H$ (e.g. see \cite{Ge-Ka-Re15a}). Moreover, it follows from \cite[Lemma 3.2.2]{Ge-Ka-Re15a} that $H\times D$ is seminormal if and only if $H$ and $D$ are seminormal. Finally, if $D$ is seminormal, then $H$ is seminormal. (This can be proved along similar lines as \cite[Lemma 3.2.4]{Ge-Ka-Re15a}.) We use these facts about seminormality without further mention.

\smallskip
\begin{theorem}\label{3.4}
Let $G$ be an abelian group. Then the following statements are equivalent.
\begin{enumerate}
\item[(a)] $\mathcal{B}_{\pm}(G)$ is a Mori monoid.
\item[(b)] $(\mathcal{B}_{\pm}(G)\colon\widehat{\mathcal{B}_{\pm}(G)})\not=\emptyset$.
\item[(c)] $2G$ is finite.
\item[(d)] $G=G_1\oplus G_2$, where $G_1$ is an elementary $2$-group and $G_2$ is a finite group.
\end{enumerate}
If these equivalent conditions are satisfied, then $\mathcal{B}_{\pm}(G)$ is seminormal if and only if ${\rm exp}(G)\mid 4$.
\end{theorem}

\begin{proof}
Let $Q$ be the quotient group of $\mathcal{B}_{\pm}(G)$.

\smallskip
(a) $\Longrightarrow$ (c) It follows from Lemma~\ref{3.3} that $G$ is a torsion group. Assume that $2G$ is infinite. Clearly, there is some $e_0\in G$ such that $2e_0\not=0$. Now let $i\in\mathbb{N}_0$ and let $(e_j)_{j=1}^i$ be elements of $G$ such that $2e_k\not\in\langle\{e_j\colon j\in [0,k-1]\}\rangle$ for each $k\in [1,i]$. Note that $\langle\{e_j\colon j\in [0,i]\}\rangle$ is finite, and hence there is some $e_{i+1}\in G$ such that $2e_{i+1}\not\in\langle\{e_j\colon j\in [0,i]\}\rangle$. Consequently, there exists a sequence $(e_i)_{i\in\mathbb{N}_0}$ of elements of $G$ such that $2e_0\not=0$ and for each $i\in\mathbb{N}_0$, $2e_{i+1}\not\in\langle\{e_j\colon j\in [0,i]\}\rangle$.

For each $n\in\mathbb{N}$ let $S_n=(2e_0)(e_0+e_n)(e_0-e_n)$ and $a_n=\frac{\prod_{i=1}^n e_i^2}{2e_0}$. It is sufficient to show that for each $n\in\mathbb{N}$, $S_n\in\mathcal{B}_{\pm}(G)$, $a_n\in Q$, $a_nS_{n+1}\not\in\mathcal{B}_{\pm}(G)$ and for each $i\in [1,n]$, $a_nS_i\in\mathcal{B}_{\pm}(G)$. (Then $(\{S_i\colon i\in [1,n]\})_v\subsetneq (\{S_i\colon i\in [1,n+1]\})_v$ for each $n\in\mathbb{N}$, a contradiction.) Let $n\in\mathbb{N}$. Since $(-1)(2e_0)+(e_0+e_n)+(e_0-e_n)=0$, we have that $S_n\in\mathcal{B}_{\pm}(G)$. Clearly, $(2e_0)^2,e_0^2(2e_0)\in\mathcal{B}_{\pm}(G)$, and thus $\frac{e_0^2}{2e_0}=\frac{e_0^2(2e_0)}{(2e_0)^2}\in Q$. Since $e_j^2\in\mathcal{B}_{\pm}(G)$ for each $j\in [0,n]$, we infer that $a_n=\frac{\prod_{i=1}^n e_i^2}{e_0^2}\frac{e_0^2}{2e_0}\in Q$. Let $i\in [1,n]$. Then $\sum_{j=1,j\not=i}^n e_j+\sum_{j=1,j\not=i}^n (-1)e_j+(-1)e_i+(-1)e_i+(e_0+e_i)+(-1)(e_0-e_i)=0$, and hence $a_nS_i\in\mathcal{B}_{\pm}(G)$.

Assume that $a_nS_{n+1}\in\mathcal{B}_{\pm}(G)$. Then there are some $(\alpha_j)_{j=1}^n,(\beta_j)_{j=1}^n\in\{-1,1\}^n$ and $\gamma,\delta\in\{-1,1\}$ such that $\sum_{j=1}^n (\alpha_j+\beta_j)e_j+\gamma(e_0+e_{n+1})+\delta(e_0-e_{n+1})=0$. If $\gamma\not=\delta$, then $2e_{n+1}\in\langle\{e_j\colon j\in [0,n]\}\rangle$, a contradiction. Therefore $\gamma=\delta$ and $\sum_{j=1}^n (\alpha_j+\beta_j)e_j+2\gamma e_0=0$. Assume that $\alpha_j+\beta_j\not=0$ for some $j\in [1,n]$. Let $j\in [1,n]$ be maximal with $\alpha_j+\beta_j\not=0$. Then $2e_j\in\langle\{e_i\colon i\in [0,j-1]\}\rangle$, a contradiction. Consequently, $\alpha_j+\beta_j=0$ for all $j\in [1,n]$, and thus $2e_0=0$, a contradiction.

\smallskip
(b) $\Longrightarrow$ (c) First we show that $2G\subset\widehat{\mathcal{B}_{\pm}(G)}$. Let $z\in 2G$. Then $z=2g$ for some $g\in G$. Note that $g^2,zg^2\in\mathcal{B}_{\pm}(G)$, and hence $z=\frac{zg^2}{g^2}\in Q$. Since $z^2\in\mathcal{B}_{\pm}(G)$, we infer that $z\in\widehat{\mathcal{B}_{\pm}(G)}$. There is some $S\in\mathcal{B}_{\pm}(G)$ such that $Sz\in\mathcal{B}_{\pm}(G)$ for each $z\in 2G$. Observe that $2G\subset\sigma_{\pm}(S)$, and thus $2G$ is finite.

\smallskip
(c) $\Longrightarrow$ (d) If $N=|2G|$, then $(2N)g=0$ for each $g\in G$. Thus $G$ is bounded, whence it is a direct sum of cyclic groups (see \cite[Chapter 4]{Ro96}). Therefore there is a set $\mathcal{I}$, a family $(G_i)_{i\in\mathcal{I}}$ of subgroups of $G$ and $(n_i)_{i\in\mathcal{I}}\in (\mathbb{N}_{\geq 2})^{\mathcal{I}}$ such that $G_i$ is cyclic of order $n_i$ for each $i\in\mathcal{I}$ and $G=\bigoplus_{i\in\mathcal{I}} G_i$. Let $\mathcal{J}=\{i\in\mathcal{I}\colon n_i=2\}$, $\mathcal{K}=\{i\in\mathcal{I}\colon n_i\not=2\}$, $G_1=\bigoplus_{i\in\mathcal{J}} G_i$ and $G_2=\bigoplus_{i\in\mathcal{K}} G_i$. Note that $G_1$ and $G_2$ are subgroups of $G$, $G=G_1\oplus G_2$ and $G_1$ is an elementary $2$-group. Moreover, since $2G$ is finite, we have that $\mathcal{K}$ is finite, and thus $G_2$ is finite.

\smallskip
(d) $\Longrightarrow$ (a) Let
\[
\varphi\colon\mathcal{B}_{\pm}(G)\rightarrow\mathcal{F}(G)\times\mathcal{B}(G_1)\times\mathcal{B}_{\pm}(G_2)\textnormal{ be defined by }
\]

\[
\varphi\left(\prod_{i=1}^r (x_i^{\prime}+x_i^{\prime\prime})\right)=\left(\prod_{i=1}^r (x_i^{\prime}+x_i^{\prime\prime}),\prod_{i=1}^r x_i^{\prime},\prod_{i=1}^r x_i^{\prime\prime}\right)
\]

for each $r\in\mathbb{N}_0$, $(x_i^{\prime})_{i=1}^r\in G_1^r$ and $(x_i^{\prime\prime})_{i=1}^r\in G_2^r$.

We prove that $\varphi$ is a divisor homomorphism. Let $r\in\mathbb{N}_0$, $(x_i^{\prime})_{i=1}^r\in G_1^r$ and $(x_i^{\prime\prime})_{i=1}^r\in G_2^r$ be such that $\prod_{i=1}^r (x_i^{\prime}+x_i^{\prime\prime})\in\mathcal{B}_{\pm}(G)$. Then there is some $(\alpha_i)_{i=1}^r\in\{-1,1\}^r$ with $\sum_{i=1}^r\alpha_i(x_i^{\prime}+x_i^{\prime\prime})=0$, and thus $\sum_{i=1}^r x_i^{\prime}=\sum_{i=1}^r\alpha_ix_i^{\prime}=0$ and $\sum_{i=1}^r\alpha_ix_i^{\prime\prime}=0$. Therefore $(\prod_{i=1}^r (x_i^{\prime}+x_i^{\prime\prime}),\prod_{i=1}^r x_i^{\prime},\prod_{i=1}^r x_i^{\prime\prime})\in\mathcal{F}(G)\times\mathcal{B}(G_1)\times\mathcal{B}_{\pm}(G_2)$. This implies that $\varphi$ is well-defined, since each element of $\mathcal{B}_{\pm}(G)$ has a unique representation (up to order) as a formal product of sums of elements of $G_1$ and $G_2$. It is straightforward to prove that $\varphi$ is a monoid homomorphism.

Let $S,T\in\mathcal{B}_{\pm}(G)$, $A\in\mathcal{F}(G)$, $B\in\mathcal{B}(G_1)$ and $C\in\mathcal{B}_{\pm}(G_2)$ be such that $\varphi(T)=\varphi(S)(A,B,C)$ (i.e., $\varphi(S)$ divides $\varphi(T)$ in $\mathcal{F}(G)\times\mathcal{B}(G_1)\times\mathcal{B}_{\pm}(G_2)$). There are some $m,n\in\mathbb{N}_0$, $(g_i^{\prime})_{i=1}^n\in G_1^n$, $(g_i^{\prime\prime})_{i=1}^n\in G_2^n$, $(h_j^{\prime})_{j=1}^m\in G_1^m$ and $(h_j^{\prime\prime})_{j=1}^m\in G_2^m$ such that $S=\prod_{i=1}^n (g_i^{\prime}+g_i^{\prime\prime})$ and $A=\prod_{j=1}^m (h_j^{\prime}+h_j^{\prime\prime})$.

We have that $T=SA$, and hence $(T,\prod_{i=1}^n g_i^{\prime}\prod_{j=1}^m h_j^{\prime},\prod_{i=1}^n g_i^{\prime\prime}\prod_{j=1}^m h_j^{\prime\prime})=\varphi(T)=\varphi(S)(A,B,C)=(SA,(\prod_{i=1}^n g_i^{\prime})B,(\prod_{i=1}^n g_i^{\prime\prime})C)$. It follows that $\prod_{j=1}^m h_j^{\prime}=B\in\mathcal{B}(G_1)$ and $\prod_{j=1}^m h_j^{\prime\prime}=C\in\mathcal{B}_{\pm}(G_2)$. Therefore $\sum_{j=1}^m h_j^{\prime}=0$ and $\sum_{j=1}^m\beta_jh_j^{\prime\prime}=0$ for some $(\beta_j)_{j=1}^m\in\{-1,1\}^m$. Note that $\sum_{j=1}^m\beta_j(h_j^{\prime}+h_j^{\prime\prime})=\sum_{j=1}^m\beta_jh_j^{\prime}+\sum_{j=1}^m\beta_jh_j^{\prime\prime}=\sum_{j=1}^m h_j^{\prime}=0$, and thus $A\in\mathcal{B}_{\pm}(G)$ and $S$ divides $T$ in $\mathcal{B}_{\pm}(G)$. This shows that $\varphi$ is a divisor homomorphism.

Clearly, $\mathcal{F}(G)$ and $\mathcal{B}(G_1)$ are Mori monoids (since they are Krull monoids). It follows from \cite[Theorem 5.1]{Ge-HK-Zh22} and \cite[Theorem 2.9.13]{Ge-HK06a} that $\mathcal{B}_{\pm}(G_2)$ is a Mori monoid. Therefore $\mathcal{F}(G)\times\mathcal{B}(G_1)\times\mathcal{B}_{\pm}(G_2)$ is a Mori monoid by \cite[Proposition 2.1.11]{Ge-HK06a}. We infer by \cite[Proposition 2.4.4.(b)]{Ge-HK06a} that $\mathcal{B}_{\pm}(G)$ is a Mori monoid.

\smallskip
(d) $\Longrightarrow$ (b) Since $G_2$ is finite, $(\mathcal{B}_{\pm}(G_2)\colon\widehat{\mathcal{B}_{\pm}(G_2)})\not=\emptyset$ by \cite[Theorem 5.1]{Ge-HK-Zh22} and \cite[Theorem 2.9.11]{Ge-HK06a}. There is some $a\in (\mathcal{B}_{\pm}(G_2)\colon\widehat{\mathcal{B}_{\pm}(G_2)})$. It is sufficient to show that $a\in (\mathcal{B}_{\pm}(G)\colon\widehat{\mathcal{B}_{\pm}(G)})$. Note that $a\in\mathcal{B}_{\pm}(G_2)\subset\mathcal{B}_{\pm}(G)$ and $a\widehat{\mathcal{B}_{\pm}(G_2)}\subset\mathcal{B}_{\pm}(G_2)$. Let $x\in\widehat{\mathcal{B}_{\pm}(G)}$. It remains to show that $ax\in\mathcal{B}_{\pm}(G)$.

\smallskip
There is some $y\in\mathcal{B}_{\pm}(G)$ such that $xy\in\mathcal{B}_{\pm}(G)$. Furthermore, there are some $\ell,m,n\in\mathbb{N}_0$, $(a_k)_{k=1}^{\ell}\in G_2^{\ell}$, $(x_i)_{i=1}^n\in G^n$ and $(y_j)_{j=1}^m\in G^m$ such that $a=\prod_{k=1}^{\ell} a_k$, $x=\prod_{i=1}^n x_i$ and $y=\prod_{j=1}^m y_j$. Finally, there are some $(x_i^{\prime})_{i=1}^n\in G_1^n$, $(x_i^{\prime\prime})_{i=1}^n\in G_2^n$, $(y_j^{\prime})_{j=1}^m\in G_1^m$ and $(y_j^{\prime\prime})_{j=1}^m\in G_2^m$ such that $x_i=x_i^{\prime}+x_i^{\prime\prime}$ for each $i\in [1,n]$ and $y_j=y_j^{\prime}+y_j^{\prime\prime}$ for each $j\in [1,m]$.

Since $y\in\mathcal{B}_{\pm}(G)$, there is some $(\alpha_j)_{j=1}^m\in\{-1,1\}^m$ with $\sum_{j=1}^m\alpha_jy_j=0$. Since $G=G_1\oplus G_2$ and $G_1$ is an elementary $2$-group, this implies that $\sum_{j=1}^m y_j^{\prime}=\sum_{j=1}^m\alpha_jy_j^{\prime}=0$ and $\sum_{j=1}^m\alpha_jy_j^{\prime\prime}=0$. Consequently, $\prod_{j=1}^m y_j^{\prime\prime}\in\mathcal{B}_{\pm}(G_2)$.

Since $xy\in\mathcal{B}_{\pm}(G)$, there are some $(\beta_i)_{i=1}^n\in\{-1,1\}^n$ and $(\gamma_j)_{j=1}^m\in\{-1,1\}^m$ with $\sum_{i=1}^n\beta_ix_i+\sum_{j=1}^m\gamma_jy_j=0$. Again since $G=G_1\oplus G_2$ and $G_1$ is an elementary $2$-group, we have that $\sum_{i=1}^n x_i^{\prime}=\sum_{i=1}^n x_i^{\prime}+\sum_{j=1}^m y_j^{\prime}=\sum_{i=1}^n\beta_ix_i^{\prime}+\sum_{j=1}^m\gamma_jy_j^{\prime}=0$ and $\sum_{i=1}^n\beta_ix_i^{\prime\prime}+\sum_{j=1}^m\gamma_jy_j^{\prime\prime}=0$. This implies that $\prod_{i=1}^n x_i^{\prime\prime}\prod_{j=1}^m y_j^{\prime\prime}\in\mathcal{B}_{\pm}(G_2)$.

Observe that $\prod_{i=1}^n x_i^{\prime\prime}\in\widehat{\mathcal{B}_{\pm}(G_2)}$ by the proof of Theorem~\ref{3.2}, since $\prod_{i=1}^n x_i^{\prime\prime}\in\mathcal{F}(G_2)$, $\prod_{j=1}^m y_j^{\prime\prime}\in\mathcal{B}_{\pm}(G_2)$ and $\prod_{i=1}^n x_i^{\prime\prime}\prod_{j=1}^m y_j^{\prime\prime}\in\mathcal{B}_{\pm}(G_2)$. It follows that $a\prod_{i=1}^n x_i^{\prime\prime}\in\mathcal{B}_{\pm}(G_2)$, and hence there are some $(\delta_k)_{k=1}^{\ell}\in\{-1,1\}^{\ell}$ and $(\varepsilon_i)_{i=1}^n\in\{-1,1\}^n$ with $\sum_{k=1}^{\ell}\delta_ka_k+\sum_{i=1}^n\varepsilon_ix_i^{\prime\prime}=0$.

Since $G_1$ is an elementary $2$-group, we infer that $\sum_{k=1}^{\ell}\delta_ka_k+\sum_{i=1}^n\varepsilon_ix_i=\sum_{k=1}^{\ell}\delta_ka_k+\sum_{i=1}^n\varepsilon_ix_i^{\prime\prime}+\sum_{i=1}^n\varepsilon_ix_i^{\prime}=\sum_{i=1}^n\varepsilon_ix_i^{\prime}=\sum_{i=1}^n x_i^{\prime}=0$. Therefore $ax\in\mathcal{B}_{\pm}(G)$.

\smallskip
Now let the equivalent conditions be satisfied and let $\varphi\colon\mathcal{B}_{\pm}(G)\rightarrow\mathcal{F}(G)\times\mathcal{B}(G_1)\times\mathcal{B}_{\pm}(G_2)$ be the divisor homomorphism from above. Furthermore, let $\psi\colon\mathcal{B}_{\pm}(G_2)\rightarrow\mathcal{F}(G_2)\times\mathcal{B}_{\pm}(G)$ be defined by $\psi(S)=(S,S)$ for each $S\in\mathcal{B}_{\pm}(G_2)$. Then $\psi$ is a divisor homomorphism (e.g. see the proof of Corollary~\ref{3.5} below).

\smallskip
First let $\mathcal{B}_{\pm}(G)$ be seminormal. Since $\mathcal{F}(G_2)$ is seminormal, $\mathcal{F}(G_2)\times\mathcal{B}_{\pm}(G)$ is seminormal, and hence $\mathcal{B}_{\pm}(G_2)$ is seminormal (since $\psi$ is a divisor homomorphism). Consequently, ${\rm exp}(G_2)\mid 4$ by \cite[Theorem 5.3.2]{Ge-HK-Zh22}. Since ${\rm exp}(G_1)\mid 2$, we obtain that ${\rm exp}(G)={\rm lcm}({\rm exp}(G_1),{\rm exp}(G_2))\mid 4$.

\smallskip
Now let ${\rm exp}(G)\mid 4$. Then ${\rm exp}(G_2)\mid 4$ and $\mathcal{B}_{\pm}(G_2)$ is seminormal by \cite[Theorem 5.3.2]{Ge-HK-Zh22}. Since $\mathcal{F}(G)$ and $\mathcal{B}(G_1)$ are Krull monoids (and thus seminormal), we have that $\mathcal{F}(G)\times\mathcal{B}(G_1)\times\mathcal{B}_{\pm}(G_2)$ is seminormal. Therefore $\mathcal{B}_{\pm}(G)$ is seminormal (since $\varphi$ is a divisor homomorphism).
\end{proof}

A monoid homomorphism $\theta\colon H\to B$ is said to be a {\it transfer homomorphism} if the following two conditions hold.
\begin{enumerate}
\item[{\bf (T\,1)\,}] $B=\theta(H) B^\times$ and $\theta^{-1}(B^\times)=H^\times$.

\item[{\bf (T\,2)\,}] If $u\in H$,\ $b,\,c\in B$ and $\theta(u)=bc$, then there exist\ $v,\,w\in H$\ such that\ $u=vw$, $\theta(v)\in bB^{\times}$, and $\theta(w)\in cB^{\times}$.
\end{enumerate}
A monoid is said to be {\it transfer Krull} if it has a transfer homomorphism to a Krull monoid.
Thus, every Krull monoid is transfer Krull, because the identity is a transfer homomorphism. For a list of transfer Krull monoids, that are not Krull, we refer to \cite[Section 5]{Ge-Zh20a} and to \cite{Ba-Re22a}.

\smallskip
\begin{corollary}\label{3.5}
Let $G$ be an abelian group. Then the following statements are equivalent.
\begin{enumerate}
\item[(a)] $\mathcal{B}_{\pm}(G)$ is a Krull monoid.
\item[(b)] $\mathcal{B}_{\pm}(G)$ is completely integrally closed.
\item[(c)] $\mathcal{B}_{\pm}(G)$ is root closed.
\item[(d)] $\mathcal{B}_{\pm}(G)$ is a transfer Krull monoid.
\item[(e)] $G$ is an elementary $2$-group.
\end{enumerate}
\end{corollary}

\begin{proof}
(a) $\Longrightarrow$ (b) Every Krull monoid is completely integrally closed.

(b) $\Longrightarrow$ (c) Every completely integrally closed monoid is root closed.

(c) $\Longrightarrow$ (a) If $\mathcal{B}_{\pm}(G)$ is root closed, then $\mathcal{B}_{\pm}(G)=\widetilde{\mathcal{B}_{\pm}(G)}$ is a Krull monoid by Theorem~\ref{3.2}.1.

(a) $\Longleftrightarrow$ (d) Every Krull monoid is transfer Krull and the reverse implication was proved in \cite[Proposition 3.8]{B-M-O-S22}.

(e) $\Longrightarrow$ (a) If $G$ is an elementary $2$-group, then $\mathcal{B}_{\pm}(G)=\mathcal{B}(G)$ is a Krull monoid by \cite[Proposition 2.5.6]{Ge-HK06a}.

(a) $\Longrightarrow$ (e) It follows from Theorem~\ref{3.4} that there are some subgroups $G_1$ and $G_2$ of $G$ such that $G=G_1\oplus G_2$, $G_1$ is an elementary $2$-group and $G_2$ is finite. Let $\varphi\colon\mathcal{B}_{\pm}(G_2)\rightarrow\mathcal{F}(G_2)\times\mathcal{B}_{\pm}(G)$ be defined by $\varphi(S)=(S,S)$ for each $S\in\mathcal{B}_{\pm}(G_2)$. Clearly, $\varphi$ is a monoid homomorphism. Moreover, since $\mathcal{F}(G_2)\cap\mathcal{B}_{\pm}(G)=\mathcal{B}_{\pm}(G_2)$, we obtain that $\varphi$ is a divisor homomorphism. It follows from \cite[Proposition 2.3.7]{Ge-HK06a} that $\mathcal{F}(G_2)\times\mathcal{B}_{\pm}(G)$ is a Krull monoid, and hence $\mathcal{B}_{\pm}(G_2)$ is a Krull monoid by \cite[Proposition 2.4.4.(b)]{Ge-HK06a}. Therefore $G_2$ is an elementary $2$-group by \cite[Theorem 4.4]{Ge-HK-Zh22}, and thus $G$ is an elementary $2$-group.
\end{proof}

\smallskip
\begin{lemma}\label{3.6}
Let $G$ be an abelian group and let $G_1$ and $G_2$ be subgroups of $G$ such that $G=G_1\oplus G_2$, $G_1$ is an elementary $2$-group and $G_2$ is finite. Then $\mathsf{q}\big(\mathcal{B}_{\pm}(G)\big)/\mathsf{q}\big(\mathcal{B}(G)\big)$ is finitely generated.
\end{lemma}

\begin{proof}
Let $N=|G_2|$ and let $E=\{g^2\colon g\in G_2\}$. First we show that for each $S\in\mathcal{B}_{\pm}(G)$, there is some $e\in [E]$ such that $S\in e\mathsf{q}\big(\mathcal{B}(G)\big)$. Let $S\in\mathcal{B}_{\pm}(G)$. Then there are some $n\in\mathbb{N}_0$, $(g_i^{\prime})_{i=1}^n\in G_1^n$ and $(g_i^{\prime\prime})_{i=1}^n\in G_2^n$ such that $S=\prod_{i=1}^n (g_i^{\prime}+g_i^{\prime\prime})$. Moreover, there is some $(\alpha_i)_{i=1}^n\in\{-1,1\}^n$ such that $\sum_{i=1}^n\alpha_i(g_i^{\prime}+g_i^{\prime\prime})=0$. Observe that $\sum_{i=1}^n g_i^{\prime}=0$ and $\sum_{i=1}^n\alpha_ig_i^{\prime\prime}=0$. Set $e=\prod_{i=1,\alpha_i=1}^n (g_i^{\prime\prime})^2$. Then $e\in [E]$. Since $\sum_{i=1}^n N(g_i^{\prime}+g_i^{\prime\prime})=N\sum_{i=1}^n g_i^{\prime}+\sum_{i=1}^n Ng_i^{\prime\prime}=0$, we have that $S^N\in\mathcal{B}(G)$. Moreover, $\sum_{i=1}^n (N-1)(g_i^{\prime}+g_i^{\prime\prime})+\sum_{i=1,\alpha_i=1}^n 2g_i^{\prime\prime}=(N-1)\sum_{i=1}^n g_i^{\prime}+\sum_{i=1}^n (N-1)g_i^{\prime\prime}+\sum_{i=1,\alpha_i=1}^n (1+\alpha_i)g_i^{\prime\prime}=\sum_{i=1}^n (N-1)g_i^{\prime\prime}+\sum_{i=1}^n (1+\alpha_i)g_i^{\prime\prime}=\sum_{i=1}^n Ng_i^{\prime\prime}+\sum_{i=1}^n\alpha_ig_i^{\prime\prime}=0$. Consequently, $S^{N-1}e\in\mathcal{B}(G)$. This implies that $S=e\frac{S^N}{S^{N-1}e}\in e\mathsf{q}\big(\mathcal{B}(G)\big)$.

\smallskip
Since $E$ is finite, it is sufficient to show that $\mathsf{q}\big(\mathcal{B}_{\pm}(G)\big)/\mathsf{q}\big(\mathcal{B}(G)\big)=\langle\{y\mathsf{q}\big(\mathcal{B}(G)\big)\colon y\in E\}\rangle$. Clearly, $E\subset\mathcal{B}_{\pm}(G)$, and thus $\langle\{y\mathsf{q}\big(\mathcal{B}(G)\big)\colon y\in E\}\rangle\subset\mathsf{q}\big(\mathcal{B}_{\pm}(G)\big)/\mathsf{q}\big(\mathcal{B}(G)\big)$. Now let $x\in\mathsf{q}\big(\mathcal{B}_{\pm}(G)\big)/\mathsf{q}\big(\mathcal{B}(G)\big)$. Then there are some $S,T\in\mathcal{B}_{\pm}(G)$ such that $x=\frac{S}{T}\mathsf{q}\big(\mathcal{B}(G)\big)$. As shown before, there are some $e,f\in [E]$ such that $S\in e\mathsf{q}\big(\mathcal{B}(G)\big)$ and $T\in f\mathsf{q}\big(\mathcal{B}(G)\big)$. It follows that $S\mathsf{q}\big(\mathcal{B}(G)\big)=e\mathsf{q}\big(\mathcal{B}(G)\big)\in\langle\{y\mathsf{q}\big(\mathcal{B}(G)\big)\colon y\in E\}\rangle$ and $T\mathsf{q}\big(\mathcal{B}(G)\big)=f\mathsf{q}\big(\mathcal{B}(G)\big)\in\langle\{y\mathsf{q}\big(\mathcal{B}(G)\big)\colon y\in E\}\rangle$. This implies that $x=\frac{S}{T}\mathsf{q}\big(\mathcal{B}(G)\big)=\frac{S\mathsf{q}\left(\mathcal{B}(G)\right)}{T\mathsf{q}\left(\mathcal{B}(G)\right)}\in\langle\{y\mathsf{q}\big(\mathcal{B}(G)\big)\colon y\in E\}\rangle$.
\end{proof}

\smallskip
\begin{theorem}\label{3.7}
Let $G$ be an abelian group. Then the following statements are equivalent.
\begin{enumerate}
\item[(a)] $\mathcal{B}_{\pm}(G)$ is finitely generated.
\item[(b)] $\mathcal{B}_{\pm}(G)$ is a C-monoid defined in $\mathcal{F}(G)$.
\item[(c)] $\mathcal{B}_{\pm}(G)$ is a C-monoid.
\item[(d)] $\mathcal{B}_{\pm}(G)$ is a Mori monoid and $\mathcal{C}_v\big(\widehat{\mathcal{B}_{\pm}(G)}\big)$ is finitely generated.
\item[(e)] $G$ is finite.
\end{enumerate}
\end{theorem}

\begin{proof}
(a) $\Longrightarrow$ (d) It is an immediate consequence of \cite[Proposition 2.7.11 and Theorems 2.7.13 and 2.7.14]{Ge-HK06a} that $\mathcal{B}_{\pm}(G)$ is a Mori monoid, $\widehat{\mathcal{B}_{\pm}(G)}$ is a finitely generated Krull monoid and $\mathfrak{X}\big(\widehat{\mathcal{B}_{\pm}(G)}\big)$ is finite (since $\mathcal{B}_{\pm}(G)$ and $\widehat{\mathcal{B}_{\pm}(G)}$ are reduced). We have that $\mathcal{C}_v\big(\widehat{\mathcal{B}_{\pm}(G)}\big)=\langle\{[P]_{\mathcal{C}_v\left(\widehat{\mathcal{B}_{\pm}(G)}\right)}\colon P\in\mathfrak{X}\big(\widehat{\mathcal{B}_{\pm}(G)}\big)\}\rangle$ (since $\widehat{\mathcal{B}_{\pm}(G)}$ is a Krull monoid), and thus $\mathcal{C}_v\big(\widehat{\mathcal{B}_{\pm}(G)}\big)$ is finitely generated.

\smallskip
(b) $\Longrightarrow$ (c) This is obvious.

\smallskip
(c) $\Longrightarrow$ (d) We have that $\mathcal{B}_{\pm}(G)$ is a Mori monoid by \cite[Theorem 2.9.13]{Ge-HK06a}. Moreover, $\mathcal{C}_v\big(\widehat{\mathcal{B}_{\pm}(G)}\big)$ is finite by \cite[Theorem 2.9.11]{Ge-HK06a}.

\smallskip
(d) $\Longrightarrow$ (e) Without restriction, we can assume that $|G|\geq 3$. It follows from Lemmas~\ref{3.3} and~\ref{3.6} and Theorem~\ref{3.4} that $G$ is a torsion group and $\mathsf{q}\big(\mathcal{B}_{\pm}(G)\big)/\mathsf{q}\big(\mathcal{B}(G)\big)$ is finitely generated. Since $\widehat{\mathcal{B}_{\pm}(G)}\hookrightarrow\mathcal{F}(G)$ is a divisor theory by Theorem~\ref{3.2}.2, we infer by \cite[Theorem 2.4.7]{Ge-HK06a} that $\mathcal{C}_v\big(\widehat{\mathcal{B}_{\pm}(G)}\big)\cong\mathsf{q}\big(\mathcal{F}(G)\big)/\mathsf{q}\big(\widehat{\mathcal{B}_{\pm}(G)}\big)$. Therefore $\mathsf{q}\big(\mathcal{F}(G)\big)/\mathsf{q}\big(\mathcal{B}_{\pm}(G)\big)=\mathsf{q}\big(\mathcal{F}(G)\big)/\mathsf{q}\big(\widehat{\mathcal{B}_{\pm}(G)}\big)$ is finitely generated, and hence $\mathsf{q}\big(\mathcal{F}(G)\big)/\mathsf{q}\big(\mathcal{B}(G)\big)$ is finitely generated. Since $G\cong\mathsf{q}\big(\mathcal{F}(G)\big)/\mathsf{q}\big(\mathcal{B}(G)\big)$ by \cite[Proposition 2.5.6]{Ge-HK06a}, we obtain that $G$ is finitely generated. Consequently, $G$ is finite (since $G$ is a torsion group).

\smallskip
(e) $\Longrightarrow$ (a),(b) This follows from \cite[Theorem 5.1]{Ge-HK-Zh22} and its proof.
\end{proof}

\smallskip
\section{On the Isomorphism Problem and the Characterization Problem}\label{4}
\smallskip

In this section, we first give an affirmative answer to the Isomorphism Problem for groups which are direct sums of cyclic groups (Theorem~\ref{4.3}). Then we study the Characterization Problem (Theorems~\ref{4.5} and~\ref{4.6}).

\smallskip
\begin{proposition}\label{4.1}
Let $G_1$ and $G_2$ be abelian groups such that $|G_1|,|G_2|\not=2$ and let $\varphi\colon\mathcal{B}_{\pm}(G_1)\to\mathcal{B}_{\pm}(G_2)$ be a monoid isomorphism.
\begin{enumerate}
\item $\varphi(0)=0$ and $|A|=|\varphi(A)|$ for every $A\in\mathcal{B}_{\pm}(G_1)$.

\item For every $g\in G_1$, there exists $h\in G_2$ with $\ord(h)=\ord(g)$ such that $\varphi(g^2)=h^2$.

\item For every $h\in G_2$, there exists $g\in G_1$ such that $\varphi(g^2)=h^2$.

\item Let $g\in G_1$. For every $k\in\Z\setminus\{0\}$, there exist $h\in G_2$ and $\varepsilon\in\{-1,1\}$ such that $\varphi((kg)^2)=(\varepsilon kh)^2$.

\item There is a bijection $\varphi_0\colon G_1\to G_2$.
\end{enumerate}
\end{proposition}

\begin{proof}
1. We first show that $\varphi(0)=0$. Assume to the contrary that there exists $U\in\mathcal{A}\big(\mathcal{B}_{\pm}(G_2)\big)$ with $U\neq 0$ such that $\varphi(0)=U$. Then $0\not\in\supp(U)$ and $|U|\ge 2$. Suppose $U=g_1\ldots g_{\ell}$. Then there exist nontrivial $T_1,\ldots,T_{\ell}\in\mathcal{B}_{\pm}(G_1)$ such that $\varphi(T_i)=g_i^2$, whence $\varphi(T_1\ldots T_{\ell})=U^2=\varphi(0^2)$. Thus $0^2=T_1\ldots T_{\ell}$, whence $\ell=2$, $T_1=T_2=0$, and $U=g_1^2$. Let $g\in G_2\setminus\{0,-g_1\}$ and let $V_1,V_2,V\in\mathcal{A}\big(\mathcal{B}_{\pm}(G_1)\big)$ such that $\varphi(V_1)=g^2$, $\varphi(V_2)=(g_1+g)^2$, $\varphi(V)=g_1g(g_1+g)$. Then $\varphi(0V_1V_2)=g_1^2g^2(g_1+g)^2=(g_1g(g_1+g))^2=\varphi(V^2)$, whence $0V_1V_2=V^2$ and hence $0\t V$. It follows from $V\in\mathcal{A}\big(\mathcal{B}_{\pm}(G_1)\big)$ that $V=0$, a contradiction.

Let $A=0^kB$ be such that $k\in\N_0$ and $0\not\in\supp(B)$. Then $|B|=\max\mathsf{L}(B^2)=\max\mathsf{L}(\varphi(B^2))=\max\mathsf{L}((\varphi(B))^2)=|\varphi(B)|$, whence $|A|=k+|B|=k+|\varphi(B)|=|\varphi(A)|$.

\smallskip
2. Let $g\in G_1$. If $g=0$, then the assertion is trivial. Suppose $g\neq 0$. Since $g^2\in\mathcal{B}_{\pm}(G_1)$, we have $\varphi(g^2)\in\mathcal{B}_{\pm}(G_2)$ and $|\varphi(g^2)|=2$. Thus there exists $h\in G_2$ such that $\varphi(g^2)\in\{h^2,h(-h)\}$. Assume to the contrary that $\varphi(g^2)=h(-h)$ with $\ord(h)\ge 3$. Then there exist nontrivial $T_1,T_2\in\mathcal{B}_{\pm}(G_1)$ such that $\varphi(T_1)=h^2$ and $\varphi(T_2)=(-h)^2$, whence $\varphi(T_1T_2)=h^2(-h)^2=\varphi(g^2)^2=\varphi(g^4)$. It follows that $T_1=T_2=g^2$ and $h=-h$, a contradiction to the assumption that $\ord(h)\ge 3$.

It remains to show that $\ord(h)=\ord(g)$. We distinguish three cases.

\smallskip
\noindent
CASE 1: $\ord(g)$ is odd.

Then $g^{\ord(g)}$ is an atom and $\varphi(g^{\ord(g)})^2=\varphi(g^{2\ord(g)})=h^{2\ord(g)}$, whence $\varphi(g^{\ord(g)})=h^{\ord(g)}\in\mathcal{A}\big(\mathcal{B}_{\pm}(G_2)\big)$. It follows that $\ord(h)=\ord(g)$.

\smallskip
\noindent
CASE 2: $\ord(g)=2m$ for some $m\in\N$.

Then $(mg)g^m$ is an atom and $\varphi((mg)g^m)^2=\varphi((mg)^2)\varphi(g^2)^m=h_0^2h^{2m}$ for some $h_0\in G_2$, whence $\varphi((mg)g^m)=h_0h^m\in\mathcal{A}\big(\mathcal{B}_{\pm}(G_2)\big)$. It follows that $h_0\in\{mh,-mh\}$. Suppose $\ord(h_0)\ge 3$. Then there exist $g'\in G_1$ with $g'\neq mg$ and $T\in\mathcal{B}_{\pm}(G_1)$ such that $\varphi((g')^2)=(-h_0)^2$ and $\varphi(T)=h_0(-h_0)$. Then $\varphi((mg)^2g'^2)=h_0^2(-h_0)^2=(h_0(-h_0))^2=\varphi(T^2)$, whence $T=(mg)g'\in\mathcal{B}_{\pm}(G_1)$. Note that $2mg=0$. We have $g'=mg$, a contradiction. Suppose $\ord(h_0)=2$. Then $h_0=mh$ and $\ord(h)\t 2m=\ord(g)$. If $\ord(h)<2m$, then $\ord(h)\le m$ and hence $(mh)h^m=(mh)h^{m-\ord(h)}\cdot h^{\ord(h)}$ is not an atom, a contradiction. Thus $\ord(h)=2m=\ord(g)$.

\smallskip
\noindent
CASE 3: $\ord(g)=\infty$.

Then for every $k\in\N$, we have $(kg)g^k\in\mathcal{A}\big(\mathcal{B}_{\pm}(G)\big)$. Assume to the contrary that $\ord(h)=n$ is finite. Then $\varphi((ng)g^n)^2=\varphi((ng)^2g^{2n})=h_0^2(h^n)^2$ for some $h_0\in G_2$, whence $\varphi((ng)g^n)=h_0h^n\in\mathcal{A}\big(\mathcal{B}_{\pm}(G_2)\big)$ and hence $h_0\in\{nh,-nh\}=\{0\}$, a contradiction.

\smallskip
3. Note that $\varphi^{-1}\colon\mathcal{B}_{\pm}(G_2)\rightarrow\mathcal{B}_{\pm}(G_1)$ is a monoid isomorphism. Let $h\in G_2$. Then 2. implies that there exists $g\in G_1$ such that $\varphi^{-1}(h^2)=g^2$ and hence $\varphi(g^2)=h^2$.

\smallskip
4. Let $g\in G_1$. Then 2. implies that there exists $h\in G_2$ with $\ord(h)=\ord(g)$ such that $\varphi(g^2)=h^2$. Let $k\in\Z\setminus\{0\}$. We set $k'=|k|$ if $\ord(g)$ is infinite and set $k'=\min\{k_1,\ord(g)-k_1\}$, where $k_1\in [0,\ord(g)-1]$ with $k_1\equiv k\mod {\ord(g)}$, if $\ord(g)$ is finite. Then $(kg)g^{k'}\in\mathcal{A}\big(\mathcal{B}_{\pm}(G_1)\big)$. Let $h_0\in G_2$ be such that $\varphi((kg)^2)=h_0^2$. Then $\varphi((kg)g^{k'})^2=\varphi((kg)^2g^{2k'})=h_0^2h^{2k'}=(h_0h^{k'})^2$, whence $\varphi((kg)g^{k'})=h_0h^{k'}$ is an atom. It follows that $h_0\in\{k'h,-k'h\}=\{kh,-kh\}$, whence there exists $\varepsilon\in\{-1,1\}$ such that $h_0=\varepsilon kh$ and $\varphi((kg)^2)=(\varepsilon kh)^2$.

\smallskip
5. An isomorphism $\mathcal{B}_{\pm}(G_1)\to\mathcal{B}_{\pm}(G_2)$ lifts to an isomorphism $\widetilde{\mathcal{B}_{\pm}(G_1)}\to\widetilde{\mathcal{B}_{\pm}(G_2)}$. Since the inclusions $\widetilde{\mathcal{B}_{\pm}(G_1)}\hookrightarrow\mathcal{F}(G_1)$ and $\widetilde{\mathcal{B}_{\pm}(G_2)}\hookrightarrow\mathcal{F}(G_2)$ are divisor theories by Theorem~\ref{3.2}, the Uniqueness Theorem for Divisor Theories (\cite[Theorem 2.4.7]{Ge-HK06a}) shows that there is an isomorphism $\psi\colon\mathcal{F}(G_1)\to\mathcal{F}(G_2)$. Each isomorphism between two free abelian monoids stems from a bijection between the bases sets, whence the claim follows.
\end{proof}

In the next remark we provide a simple example showing that such a bijection $\varphi_0\colon G_1\to G_2$, as given above, need not be a homomorphism.

\smallskip
\begin{remark}\label{4.2}
Let $G$ be an abelian group and let $g\in G$ with $\ord(g)\ge 5$. Then the map $\varphi\colon G\rightarrow G$, defined by $\varphi(g)=-g$, $\varphi(-g)=g$, and $\varphi(h)=h$ for all $h\in G\setminus\{g,-g\}$, is a bijection. Since $\varphi(2g)=2g\neq-2g=\varphi(g)+\varphi(g)$, we observe that $\varphi$ is not a homomorphism. The bijection $\varphi$ induces a monoid isomorphism $\psi\colon\mathcal{F}(G)\rightarrow\mathcal{F}(G)$, and it is easy to see that the restriction $\psi|_{\mathcal{B}_{\pm}(G)}$ is also a monoid isomorphism. Thus, we have an isomorphism between monoids of plus-minus weighted zero-sum sequences, which does not stem from a group homomorphism.
\end{remark}

\smallskip
\begin{theorem}\label{4.3}
Let $G_1$ and $G_2$ be abelian groups and suppose that $G_1$ is a direct sum of cyclic groups. Then the groups $G_1$ and $G_2$ are isomorphic if and only their monoids of plus-minus weighted zero-sum sequences $\mathcal{B}_{\pm}(G_1)$ and $\mathcal{B}_{\pm}(G_2)$ are isomorphic.
\end{theorem}

\begin{proof}
If $G_1$ and $G_2$ are isomorphic, then the associated monoids $\mathcal{B}_{\pm}(G_1)$ and $\mathcal{B}_{\pm}(G_2)$ are isomorphic. Conversely, suppose we have a monoid isomorphism $\varphi\colon\mathcal{B}_{\pm}(G_1)\to\mathcal{B}_{\pm}(G_2)$. If one of the monoids is factorial, then both are factorial and Lemma~\ref{3.1} shows that $G_1$ and $G_2$ are isomorphic.

Suppose that none of the monoids is factorial. Then Lemma~\ref{3.1} implies that $|G_1|\ge 3$ and $|G_2|\ge 3$. Suppose that $G_1=\bigoplus_{j\in J}\langle\{e_j\}\rangle$. By Proposition~\ref{4.1}.2, there exist $f_j\in G_2$ with $\ord(f_j)=\ord(e_j)$ and $\varphi(e_j^2)=f_j^2$ for all $j\in J$. We define a group homomorphism $\psi\colon G_1\rightarrow G_2$ by setting $\psi(\sum_{i\in I}k_ie_i)=\sum_{i\in I}k_if_i$ for all finite subsets $I\subset J$ and all $k_i\in\Z$ with $i\in I$.

We first show that $\psi$ is surjective. Let $h\in G_2$. We need to show that $h\in\psi(G_1)$. By Proposition~\ref{4.1}.3, there exists $g\in G_1$ such that $\varphi(g^2)=h^2$ and $g=\sum_{j\in J_0}t_je_j$ for some finite subset $J_0\subset J$ and $t_j\in\Z\setminus\{0\}$ for all $j\in J_0$, whence $g\prod_{j\in J_0}e_j^{|t_j|}\in\mathcal{B}_{\pm}(G_1)$ and $\varphi(g^2\prod_{j\in J_0}e_j^{2|t_j|})=h^2\prod_{j\in J_0}f_j^{2|t_j|}$. It follows that $\varphi(g\prod_{j\in J_0}e_j^{|t_j|})=h\prod_{j\in J_0}f_j^{|t_j|}\in\mathcal{B}_{\pm}(G_2)$, whence $h\in\langle\{f_j\colon j\in J_0\}\rangle\subset\psi(G_1)$.

It remains to show that $\psi$ is a monomorphism. Assume to the contrary that $\psi$ is not a monomorphism. Then there exist finite $\emptyset\neq I\subset J$ and $k_i\in\Z\setminus\{0\}$ for $i\in I$ such that $\sum_{i\in I}k_if_i=0$. By Proposition~\ref{4.1}.4, there exist $\varepsilon_i\in\{-1,1\}$ for all $i\in I$ such that $\varphi((\varepsilon_ik_ie_i)^2)=(k_if_i)^2$ for all $i\in I$, whence $\varphi(\prod_{i\in I}(\varepsilon_ik_ie_i)^2)=(\prod_{i\in I}k_if_i)^2$. Let $T\in\mathcal{B}_{\pm}(G_1)$ be such that $\varphi(T)=\prod_{i\in I}k_if_i$, whence $\varphi(T^2)=\varphi(\prod_{i\in I}(\varepsilon_ik_ie_i)^2)$ and hence $T=\prod_{i\in I} (\varepsilon_ik_ie_i)\in\mathcal{B}_{\pm}(G_1)$, a contradiction to the independence of $(e_i)_{i\in I}$.
\end{proof}

\smallskip
Our next goal is to settle the Characterization Problem for cyclic groups of odd order (Theorem~\ref{4.6}; a weaker result in this direction is given by \cite[Theorem 6.10]{Ge-HK-Zh22}). In order to do so, we need some more invariants controlling the structure of sets of lengths.

Let $H$ be a BF-monoid. Then
\[
\Delta(H)=\bigcup_{L\in\mathcal{L}(H)}\Delta(L)\ \subset\N
\]
denotes the {\it set of distances} of $H$. By definition, we have that $H$ is half-factorial if and only if $\Delta(H)=\emptyset$. If $H$ is not half-factorial, then $\min\Delta(H)=\gcd\Delta(H)$. Let $\omega(H)$ be the smallest $N\in\N_0\cup\{\infty\}$ with the following property.
\begin{itemize}
\item[] For all $u\in\mathcal{A}(H)$, all $n\in\N$ and all $a_1,\ldots,a_n\in H$ with $u\t a_1\cdot\ldots\cdot a_n$, there is $\Omega\subset [1,n]$ such that $|\Omega|\le N$ and $u\t\prod_{\nu\in\Omega} a_{\nu}$.
\end{itemize}
If $H$ is not half-factorial, then, by \cite[Proposition 3.6.3]{Ge-Ka10a}, we have
\begin{equation}\label{catenary}
2+\sup\Delta(H)\le\omega(H)\,.
\end{equation}
A subset $L\subset\Z$ is said to be an {\it almost arithmetic progression} (AAP) with difference $d\in\N$, length $\ell$, and bound $M$ if
\[
L=y+(L'\cup L^*\cup L'')\subset y+d\Z\,,
\]
where $L^*$ is an arithmetic progression with difference $d$, length $\ell$, and $\min L^*=0$, $L'\subset [-M,-1]$, and $L''\subset\max L^*+[1,M]$. We define $\Delta_1(H)$ to be the set of all $d\in\N$ having the following property:
\begin{itemize}
\item[] For every $k\in\N$, there is some $L_k\in\mathcal{L}(H)$ that is an AAP with difference $d$ and length at least $k$.
\end{itemize}

For $k\in\N$, we denote by
\begin{itemize}
\item $\mathcal{U}_k(H)=\bigcup_{k\in L,L\in\mathcal{L}(H)}\ L\ \subset\N $ the {\it union of sets of lengths} (containing $k$), and by
\item $\rho_k(H)=\sup\mathcal{U}_k(H)$ the $k$th {\it elasticity} of $H$.
\end{itemize}
The unions $\mathcal{U}_k\big(\mathcal{B}_{\pm}(G)\big)$ are finite intervals by \cite[Theorem 5.2]{B-M-O-S22} and for the elasticity $\rho(H)$ we have
\[
\rho(H)=\sup\Big\{\frac{\max L}{\min L}\colon\{0\}\ne L\in\mathcal{L}(H)\Big\}=\lim_{k\to\infty}\frac{\rho_k(H)}{k}\,.
\]

\smallskip
\begin{lemma}\label{4.4}
Let $G_1$ and $G_2$ be finite abelian groups such that $\mathcal{L}\big(\mathcal{B}_{\pm}(G_1)\big)=\mathcal{L}\big(\mathcal{B}_{\pm}(G_2)\big)$.
\begin{enumerate}
\item $\max\Delta_1\big(\mathcal{B}_{\pm}(G_1)\big)=\max\Delta_1\big(\mathcal{B}_{\pm}(G_2)\big)$.
\item $\rho_k\big(\mathcal{B}_{\pm}(G_1)\big)=\rho_k\big(\mathcal{B}_{\pm}(G_2)\big)$ for every $k\in\N$, and $\mathsf{D}_{\pm}(G_1)=\mathsf{D}_{\pm}(G_2)$.
\end{enumerate}
\end{lemma}

\begin{proof}
The claims on $\Delta_1(\cdot )$ and on $\rho_k(\cdot)$ follow immediately from $\mathcal{L}\big(\mathcal{B}_{\pm}(G_1)\big)=\mathcal{L}\big(\mathcal{B}_{\pm}(G_2)\big)$. Since $\rho_2\big(\mathcal{B}_{\pm}(G_i)\big)=\mathsf{D}_{\pm}(G_i)$ for $i\in [1,2]$ by \cite[Theorem 5.7]{B-M-O-S22}, we infer that $\mathsf{D}_{\pm}(G_1)=\mathsf{D}_{\pm}(G_2)$.
\end{proof}

\smallskip
Let $G=C_{n_1}\oplus\ldots\oplus C_{n_r}$ with $1<n_1\t\ldots\t n_r$. We set
\[
\mathsf{D}^*(G)=1+\sum_{i=1}^r (n_i-1)\,.
\]
Then $\mathsf{D}^*(G)\le\mathsf{D}(G)$ and equality holds if $r\le 2$ or if $G$ is a $p$-group (see \cite[Chapter 5]{Ge-HK06a}). If $|G|$ has odd order, then $\mathsf{D}_{\pm}(G)=\mathsf{D}(G)$ by \cite[Corollary 6.2]{B-M-O-S22}. Set $n_0=1$. If $n$ is even, then $\mathsf{D}_{\pm}(C_n)=1+n/2$ and if $t\in [0,r]$ is maximal such that $2\nmid n_t$, then
\[
\mathsf{D}_{\pm}(G)\ge 1+\sum_{i=1}^t (n_i-1)+\sum_{i=t+1}^r\frac{n_i}{2}
\]
(see \cite[Theorem 6.7 and Corollary 6.8]{B-M-O-S22}).
This shows that $G$ has Davenport constant $\mathsf{D}_{\pm}(G)=3$ if and only if $G$ is isomorphic to one of the groups
\[
C_3,\ C_4 ,\ C_2\oplus C_2\,.
\]
Furthermore, we have $\mathsf{D}_{\pm}(G)=4$ if and only if $G$ is either isomorphic to $C_2^3$ or to $C_2\oplus C_4$. Indeed, the above mentioned results on the Davenport constant show that no other finite abelian groups $G$ can have $\mathsf{D}_{\pm}(G)=4$ and in the following theorem we outline that $\mathsf{D}_{\pm}(C_2\oplus C_4)=4$.

\smallskip
\begin{theorem}\label{4.5}
\textnormal{ }
\begin{enumerate}
\item $\mathcal{L}\big(\mathcal{B}_{\pm}(C_1)\big)=\mathcal{L}\big(\mathcal{B}_{\pm}(C_2)\big)=\big\{\{k\}\colon k\in\N_0\big\}$.
\item $\mathcal{L}\big(\mathcal{B}_{\pm}(C_3)\big)=\mathcal{L}\big(\mathcal{B}_{\pm}(C_4)\big)=\mathcal{L}\big(\mathcal{B}_{\pm}(C_2\oplus C_2)\big)=\big\{y+2k+[0,k]\colon y,k\in\N_0\big\}$.
\item $\mathcal{L}\big(\mathcal{B}_{\pm}(C_2^3)\big)=\bigl\{y+(k+1)+[0,k]\colon\, y\in\N_0,\ k\in [0,2]\bigr\}\ \cup$\\
$\hspace*{2.2cm}\bigl\{y+k+[0,k]\colon\, y\in\N_0,\ k\ge 3\bigr\}\cup\bigl\{y+2k+2\cdot [0,k]\colon\, y ,\, k\in\N_0\bigr\}$.
\item $\mathsf{D}_{\pm}(C_2\oplus C_4)=4$, $\Delta\big(\mathcal{B}_{\pm}(C_2\oplus C_4)\big)=[1,2]$, and\\
 $\mathcal{L}\big(\mathcal{B}_{\pm}(C_2\oplus C_4)\big)=\bigl\{y+k+[0,k]\colon\, y\in\N_0,\ k\ge 2\bigr\}\cup\bigl\{y+2k+2\cdot [0,k]\colon\, y ,\, k\in\N_0\bigr\}$.
\end{enumerate}
In particular, we have $\mathcal{L}\big(\mathcal{B}_{\pm}(C_2^3)\big)\subsetneq\mathcal{L}\big(\mathcal{B}_{\pm}(C_2\oplus C_4)\big)$.
\end{theorem}

\begin{proof}
1. This follows from Lemma~\ref{3.1}.

\smallskip
2. If $g\in C_3$ with $\ord(g)=3$, then
\[
\mathcal{A}\big(\mathcal{B}_{\pm}(C_3)\big)=\big\{0,g^2,(-g)^2,(-g)g,g^3,g^2(-g),g(-g)^2,(-g)^3\big\}\,.
\]
If $g\in C_4$ with $\ord(g)=4$, then
\[
\mathcal{A}\big(\mathcal{B}_{\pm}(C_4)\big)=\big\{0,g^2,(-g)^2,(-g)g,(2g)^2,(2g)g^2,(2g)(-g)^2,(2g)(-g)g\big\}\,.
\]
If $e_1,e_2\in C_2\oplus C_2$ are distinct and nonzero, then
\[
\mathcal{A}\big(\mathcal{B}_{\pm}(C_2\oplus C_2)\big)=\big\{0,e_1^2,e_2^2,(e_1+e_2)^2,e_1e_2(e_1+e_2)\big\}\,.
\]
This shows that, in each of the three groups, the sequence $S=0$ is the only prime element and the product of any two atoms of length three has a factorization as a product of three atoms of length two. Thus, the assertion follows (details in case of $C_2\oplus C_2$ are given in \cite[Theorem 7.3.2]{Ge-HK06a}).

\smallskip
3. Since $\mathcal{B}(C_2^3)=\mathcal{B}_{\pm}(C_2^3)$, the assertion follows from \cite[Theorem 7.3.2]{Ge-HK06a}.

\smallskip
4. We set $G=C_2\oplus C_4$ and choose a basis $(e_1,e_2)$ of $G$ with $\ord(e_1)=2$ and $\ord(e_2)=4$. Then $G=\{0,e_1,2e_2,e_1+2e_2,\pm e_2,\pm(e_1+e_2)\}$. We proceed in five steps.

\begin{enumerate}
\item[{\bf A1.}] $\mathsf{D}_{\pm}(G)=4$ and $\rho\big(\mathcal{B}_{\pm}(G)\big)=2$.
\end{enumerate}

{\it Proof of {\bf A1.}}
Since $5=\mathsf{D}(G)\ge\mathsf{D}_{\pm}(G)\ge 4$, in order to show $\mathsf{D}_{\pm}(G)=4$, it suffices to prove that for every $A\in\mathcal{A}\big(\mathcal{B}(G)\big)$ with $|A|=5$, we have $A\not\in\mathcal{A}\big(\mathcal{B}_{\pm}(G)\big)$. Let $U\in\mathcal{A}\big(\mathcal{B}(G)\big)$ with $|U|=5$. The elements of $\mathcal{A}\big(\mathcal{B}(C_2\oplus C_4)\big)$ are written down explicitly in \cite[Lemma 4.6]{Ge-Sc-Zh17b}. Here we go briefly through the possible cases. By symmetry and after renumbering if necessary, we may assume that $\mathsf{v}_{e_2}(U)=\mathsf{h}(U)$. Note that $U$ has four terms of order $4$ and one term of order $2$. Moreover, $\{e_1+e_2,e_1-e_2\}\not\subset\{g\in\supp(U)\colon\ord(g)=4\}\subset\{e_2,e_1+e_2,e_1-e_2\}$. We have $\mathsf{h}(U)=3$, and hence $U=e_2^3(e_1+e_2)e_1$ or $U=e_2^3(e_1-e_2)(e_1+2e_2)$, which is not in $\mathcal{A}\big(\mathcal{B}_{\pm}(G)\big)$. Therefore $\mathsf{D}_{\pm}(G)=4$ and $\rho\big(\mathcal{B}_{\pm}(G)\big)=2$ by \cite[Theorem 5.7]{B-M-O-S22}.\qed ({\bf A1})

\begin{enumerate}
\item[{\bf A2.}] On $\mathcal{A}\big(\mathcal{B}_{\pm}(G)\big)$ and some relations.
\end{enumerate}

We set $G_0=\{0,e_1,2e_2,e_1+2e_2,e_2,e_1+e_2\}$ and observe that
\[
\mathcal{L}\big(\mathcal{B}_{\pm}(G)\big)=\mathcal{L}\big(\mathcal{B}_{\pm}(G_0)\big)\ \text{ and }\ \mathsf{D}_{\pm}(G)=\mathsf{D}_{\pm}(G_0)\,.
\]
A simple calculation shows that
\begin{itemize}
\item $\{A\in\mathcal{A}\big(\mathcal{B}_{\pm}(G_0)\big)\colon |A|=4\}=\{e_2^2e_1(e_1+2e_2),(e_1+e_2)^2e_1(e_1+2e_2),e_2(e_1+e_2)e_1(2e_2),e_2(e_1+e_2)(2e_2)(e_1+2e_2)\}$ and
\item $\{A\in\mathcal{A}\big(\mathcal{B}_{\pm}(G_0)\big)\colon |A|=3\}=\{e_1(2e_2)(e_1+2e_2),e_2^2(2e_2),(e_1+e_2)^2(2e_2),e_2(e_1+e_2)(e_1+2e_2),e_2(e_1+e_2)e_1\}$.
\end{itemize}
Note that
\begin{enumerate}
\item[(i)] For every atom $A\in\mathcal{A}\big(\mathcal{B}_{\pm}(G_0)\big)$ of length $4$, we have $\mathsf{L}_{\mathcal{B}_{\pm}(G_0)}(A^2)=\{2,4\}$ if and only if $A\in\{e_2^2e_1(e_1+2e_2),(e_1+e_2)^2e_1(e_1+2e_2)\}$;
\item[(ii)] For every atom $A\in\mathcal{A}\big(\mathcal{B}_{\pm}(G_0)\big)$ of length $4$, we have $\mathsf{L}_{\mathcal{B}_{\pm}(G_0)}(A^2)=[2,4]$ if and only if $A\in\{e_2(e_1+e_2)e_1(2e_2),e_2(e_1+e_2)(2e_2)(e_1+2e_2)\}$;
\item[(iii)] For any two distinct atoms $A_1,A_2\in\mathcal{A}\big(\mathcal{B}_{\pm}(G_0)\big)$ of length $4$, we have $e_1e_2(e_1+e_2)$ dividing $A_1A_2$ in $\mathcal{B}_{\pm}(G_0)$, which implies that $3\in\mathsf{L}_{\mathcal{B}_{\pm}(G_0)}(A_1A_2)$;
\item[(iv)] For any two atoms $A_1,A_2\in\mathcal{A}\big(\mathcal{B}_{\pm}(G_0)\big)$ of length $3$, we have either $3\in\mathsf{L}_{\mathcal{B}_{\pm}(G_0)}(A_1A_2)$ or $A_1A_2=U_1U_2$ for some atoms $U_1,U_2\in\mathcal{A}\big(\mathcal{B}_{\pm}(G_0)\big)$ with $|U_1|=2$ and $|U_2|=4$;
\item[(v)] $\Delta\big(\mathcal{B}_{\pm}(\{e_1,e_2,e_1+2e_2\})\big)=\{2\}$;
\item[(vi)] $\mathsf{L}_{\mathcal{B}_{\pm}(G_0)}(U_1^2)=\{2,4\}$, $\mathsf{L}_{\mathcal{B}_{\pm}(G_0)}(U_2^2)=[2,4]$, and $\mathsf{L}_{\mathcal{B}_{\pm}(G_0)}(U_1U_2U_3)=[3,6]$, where $U_1=e_2^2e_1(e_1+2e_2)$, $U_2=e_2(e_1+e_2)e_1(2e_2)$, and $U_3=e_2(e_1+e_2)(2e_2)(e_1+2e_2)$;
\item[(vii)] For all atoms $A\in\mathcal{A}\big(\mathcal{B}_{\pm}(G_0)\big)\setminus\mathcal{A}\big(\mathcal{B}_{\pm}(\{e_1,2e_2,e_1+2e_2\})\big)$, we have $\sigma_{\pm}(A)=\{0,2e_2\}$.
\end{enumerate}

\medskip
\begin{enumerate}
\item[{\bf A3.}] $\Delta\big(\mathcal{B}_{\pm}(G)\big)=[1,2]$.
\end{enumerate}

{\it Proof of {\bf A3.}} Relation (vi) shows that $[1,2]\subset\Delta\big(\mathcal{B}_{\pm}(G)\big)$. Thus, by the Inequality \eqref{catenary}, it suffices to verify that $\omega\big(\mathcal{B}_{\pm}(G)\big)\le 4$.

Let $A,A_1,\ldots,A_{\ell}\in\mathcal{A}\big(\mathcal{B}_{\pm}(G_0)\big)\setminus\{0\}$ such that $A\t A_1\cdot\ldots\cdot A_{\ell}$ in $\mathcal{B}_{\pm}(G_0)$. If $\ell\le 4$, then there is nothing to do. Suppose $\ell\ge 5$. Since $|A|\le 4$, after renumbering if necessary, we may assume that $A\t A_1\cdot\ldots\cdot A_4$. If $A=g_1g_2g_3g_4$ such that $g_i\t A_i$ for every $i\in [1,4]$, then $A^{-1}A_1\cdot\ldots\cdot A_4\in\mathcal{B}_{\pm}(G_0)$ and hence $A\t A_1\cdot\ldots\cdot A_4$ in $\mathcal{B}_{\pm}(G_0)$. Otherwise after renumbering if necessary we may assume that $A\t A_1A_2A_3$.
Set $A'=A^{-1}A_1 A_2 A_3$. Then (vii) implies that $\sigma_{\pm}(A' )\subset\sigma_{\pm}(A_1 A_2 A_3)=\{0,2e_2\}$. If $0\in\sigma_{\pm}(A')$, then $A\t A_1A_2A_3$ in $\mathcal{B}_{\pm}(G_0)$. Suppose $\sigma_{\pm}(A')=\{2e_2\}$. Since $A\t A_1\cdot\ldots\cdot A_{\ell}$ in $\mathcal{B}_{\pm}(G_0)$, there exists $i\in [4,\ell]$ such that $\sigma_{\pm}(A_i)=\{0,2e_2\}$, whence $0\in\sigma_{\pm}(A'A_i)$. It follows that $A\t A_1A_2A_3A_i$. Therefore $\omega\big(\mathcal{B}_{\pm}(G)\big)=\omega\big(\mathcal{B}_{\pm}(G_0)\big)\le 4$.\qed ({\bf A3})

\medskip
\begin{enumerate}
\item[{\bf A4.}] $\mathcal{L}\big(\mathcal{B}_{\pm}(G_0)\big)\supset\bigl\{y+k+[0,k]\colon\, y\in\N_0,\ k\ge 2\bigr\}\cup\bigl\{y+2k+2\cdot [0,k]\colon \, y ,\, k\in\N_0\bigr\}$.
\end{enumerate}

{\it Proof of {\bf A4.}} Let $L=y+2k+2\cdot [0,k]$ for some $y,k\in\N_0$. Then (v) implies that $\mathsf{L}_{\mathcal{B}_{\pm}(G_0)}(0^yU_1^{2k})=L$, whence $\{y+2k+2\cdot [0,k]\colon\, y ,\, k\in\N_0\bigr\}\subset\mathcal{L}\big(\mathcal{B}_{\pm}(G_0)\big)$.

Let $L=y+k+[0,k]$ for some $y\in\N_0$ and some $k\ge 2$. Suppose $k$ is even. Then (vi) implies that $\mathsf{L}_{\mathcal{B}_{\pm}(G_0)}(0^yU_2^k)=L$. Suppose $k\ge 3$ is odd. Then (v) and (vi) imply that $\mathsf{L}_{\mathcal{B}_{\pm}(G_0)}(0^yU_1^{k-2}U_2U_3)=L$, whence $\{y+k+[0,k]\colon\, y\in\N_0 ,\, k\ge 2\bigr\}\subset\mathcal{L}\big(\mathcal{B}_{\pm}(G_0)\big)$.\qed ({\bf A4})

\medskip
\begin{enumerate}
\item[{\bf A5.}] $\mathcal{L}\big(\mathcal{B}_{\pm}(G_0)\big)\subset\bigl\{y+k+[0,k]\colon\, y\in\N_0,\ k\ge 2\bigr\}\cup\bigl\{y+2k+2\cdot [0,k]\colon\, y ,\, k\in\N_0\bigr\}$.
\end{enumerate}

{\it Proof of {\bf A5.}} Let $B\in\mathcal{B}_{\pm}(G_0)$. We distinguish three cases.

\smallskip
\noindent
CASE 1: $\Delta(\mathsf{L}_{\mathcal{B}_{\pm}(G_0)}(B))=\emptyset$.

Then $\mathsf{L}_{\mathcal{B}_{\pm}(G_0)}(B)\in\bigl\{y+2k+2\cdot [0,k]\colon\, y ,\, k\in\N_0\bigr\}$.

\smallskip
\noindent
CASE 2: $2\in\Delta(\mathsf{L}_{\mathcal{B}_{\pm}(G_0)}(B))$.

We set $B=\prod_{i=1}^{r}U_i^{u_i}\prod_{j=1}^{s}V_j^{v_j}\prod_{k=1}^{t}W_k$, where $r,s,t,u_i,v_j\in\N_0$, $U_i\in\mathcal{A}\big(\mathcal{B}_{\pm}(G_0)\big),i\in [1,r]$ are distinct atoms of length $4$, $V_j\in\mathcal{A}\big(\mathcal{B}_{\pm}(G_0)\big),j\in [1,s]$ are distinct atoms of length $3$, and $W_k\in\mathcal{A}\big(\mathcal{B}_{\pm}(G_0)\big),k\in [1,t]$ are atoms of length $2$, such that $\sum_{i=1}^ru_i+\sum_{j=1}^{s}v_j+t+1\not\in\mathsf{L}_{\mathcal{B}_{\pm}(G_0)}(B)$ and $\sum_{i=1}^ru_i+\sum_{j=1}^{s}v_j+t+2\in\mathsf{L}_{\mathcal{B}_{\pm}(G_0)}(B)$. We may assume that the factorization $B=\prod_{i=1}^{r}U_i^{u_i}\prod_{j=1}^{s}V_j^{v_j}\prod_{k=1}^{t}W_k$ is the one such that $\sum_{i=1}^ru_i$ is maximal. By (iv), we obtain that $\sum_{j=1}^sv_j\le 1$, whence $s=1$ and $v_1\in\{0,1\}$. By (iii), we have $r=1$ and by (i) and (ii), we have $U_1\in\{e_2^2e_1(e_1+2e_2),(e_1+e_2)^2e_1(e_1+2e_2)\}$. After changing bases if necessary, we may assume that $U_1=e_2^2e_1(e_1+2e_2)$. Moreover, we have $u_1\ge 2$ since $u_1+v_1+t+2\in\mathsf{L}_{\mathcal{B}_{\pm}(G_0)}(B)$. If $\supp(B)=\{e_1,e_2,e_1+2e_2\}$, then $v_1=0$ and (v) implies that $\Delta(\mathsf{L}_{\mathcal{B}_{\pm}(G_0)}(B))=\{2\}$, whence $\mathsf{L}_{\mathcal{B}_{\pm}(G_0)}(B)=u_1+t+2\cdot [0,\lfloor u_1/2\rfloor]\in\bigl\{y+2k+2\cdot [0,k]\colon\, y ,\, k\in\N_0\bigr\}$. Otherwise there exists $g\in G_0\setminus\{e_1,e_2,e_1+2e_2\}$ and hence $g=e_1+e_2$ or $2e_2$. If there exists $k\in [1,t]$ such that $g\t W_k$, then $4\in\mathsf{L}_{\mathcal{B}_{\pm}(G_0)}(U_1^2W_k)$, a contradiction. Suppose $g\t V_1$. If $V_1\neq (e_1+e_2)^2(2e_2)$, then $3\in\mathsf{L}_{\mathcal{B}_{\pm}(G_0)}(U_1V_1)$, a contradiction. If $V_1=(e_1+e_2)^2(2e_2)$, then $4\in\mathsf{L}_{\mathcal{B}_{\pm}(G_0)}(U_1^2V_1)$, a contradiction.

\smallskip
\noindent
CASE 3: $1\in\Delta(\mathsf{L}_{\mathcal{B}_{\pm}(G_0)}(B))$.

Then $\Delta(\mathsf{L}_{\mathcal{B}_{\pm}(G_0)}(B))=1$ and $\min\mathsf{L}_{\mathcal{B}_{\pm}(G_0)}(B)\ge 2$, whence
\[
\mathsf{L}_{\mathcal{B}_{\pm}(G_0)}(B)=[\min\mathsf{L}_{\mathcal{B}_{\pm}(G_0)}(B),\max\mathsf{L}_{\mathcal{B}_{\pm}(G_0)}(B)]\subset [\min\mathsf{L}_{\mathcal{B}_{\pm}(G_0)}(B), 2\min\mathsf{L}_{\mathcal{B}_{\pm}(G_0)}(B)]\,.
\]
Therefore $\mathsf{L}_{\mathcal{B}_{\pm}(G_0)}(B)\in\bigl\{y+k+[0,k]\colon\, y\in\N_0,\ k\ge 2\bigr\}$.\qed ({\bf A5})
\end{proof}

\smallskip
\begin{theorem}\label{4.6}
Let $G$ be a finite abelian group and let $n\ge 5$ be odd such that $\mathcal{L}\big(\mathcal{B}_{\pm}(G)\big)=\mathcal{L}\big(\mathcal{B}_{\pm}(C_n)\big)$. Then $G\cong C_n$.
\end{theorem}

\begin{proof}
Since $n$ is odd, \cite[Corollary 6.2]{B-M-O-S22} implies that $\mathsf{D}_{\pm}(C_n)=\mathsf{D}(C_n)$, whence $n=\mathsf{D}(C_n)=\mathsf{D}_{\pm}(C_n)=\mathsf{D}_{\pm}(G)$ by Lemma~\ref{4.4}. Let $L\in\mathcal{L}\big(\mathcal{B}_{\pm}(C_n)\big)$ be such that $\{2,n\}\subset L$. Then there exist $A_1,A_2,U_1,\ldots,U_n\in\mathcal{A}\big(\mathcal{B}_{\pm}(C_n)\big)$ such that $A_1A_2=U_1\cdot\ldots\cdot U_n$ and $\mathsf{L}_{\mathcal{B}_{\pm}(C_n)}(A_1A_2)=L$. Since $|A_i|\le n$ for every $i\in [1,2]$ and $|U_j|\ge 2$ for every $j\in [1,n]$, we have that $|A_1|=|A_2|=n$ and $|U_j|=2$ for all $j\in [1,n]$. It follows that $\supp(A_1),\supp(A_2)\subset\{g,-g\}$ for some $g\in C_n$ with $\ord(g)=n$, whence $L=\{2,n\}$. Therefore, for every $L\in\mathcal{L}\big(\mathcal{B}_{\pm}(G)\big)$ such that $\{2,n\}\subset L$, we have $L=\{2,n\}$.

Let $U\in\mathcal{A}\big(\mathcal{B}_{\pm}(G)\big)$ with $|U|=\mathsf{D}_{\pm}(G)=n$. Without loss of generality, we may assume that $U\in\mathcal{A}\big(\mathcal{B}(G)\big)$. Since $\{2,n\}\subset\mathsf{L}_{\mathcal{B}_{\pm}(G)}(U^2)$, we have $\{2,n\}=\mathsf{L}_{\mathcal{B}_{\pm}(G)}(U^2)$, whence for every atom $W\in\mathcal{A}\big(\mathcal{B}_{\pm}(G)\big)$ dividing $U^2$, we obtain $|W|\in\{2,n\}$.

If $|\supp(U)|=1$, then $U=g^{n}$ for some $g\in G$ with $\ord(g)=n$ and hence $G\cong C_n$.

Suppose $|\supp(U)|\ge 2$. If there exists $g\in\supp(U)$ with $\ord(g)\ge 3$ such that $\mathsf{v}_g(U)=1$, then we set $V=g^{-1}(2g)(-g)U\in\mathcal{B}_{\pm}(G)$. Since $|V|>n$, we have that $g^{-1}U$ has a decomposition $g^{-1}U=T_1T_2$ such that $(2g)T_1,(-g)T_2\in\mathcal{B}_{\pm}(G)$. Since $gT_2\in\mathcal{B}_{\pm}(G)$ and $|gT_2|<n$, we have that $gT_2$ is a product of atoms of length $2$, a contradiction to the fact that $g\not\in\supp(T_2)$ and $-g\not\in\supp(T_2)$. Thus $\mathsf{h}(U)=1$ implies that all terms of $U$ have order $2$. If $\mathsf{h}(U)=1$, then $\langle\supp(U)\rangle=G$ implies that $G$ is an elementary $2$-group, whence $G\cong C_2^{n-1}$ and $\mathcal{B}_{\pm}(G)=\mathcal{B}(G)$. By \cite[Corollary 6.8.3]{Ge-HK06a}, we have $n-3\in\Delta_1\big(\mathcal{B}(G)\big)=\Delta_1\big(\mathcal{B}_{\pm}(C_n)\big)$. For every $k\in\N$, there exists $S_k\in\mathcal{B}_{\pm}(C_n)$ such that $\mathsf{L}_{\mathcal{B}_{\pm}(C_n)}(S_k)$ is an AAP with difference $n-3$ and length at least $k$.
Since $\mathcal{A}\big(\mathcal{B}_{\pm}(C_n)\big)$ is finite, for every large enough $k\in\N$, there exists $V\in\mathcal{A}\big(\mathcal{B}_{\pm}(C_n)\big)$ such that $V^{2n}\t S_k$ in $\mathcal{B}_{\pm}(C_n)$, whence for every $g\in C_n$, we have $g^{2\ord(g)}\t S_k$ in $\mathcal{B}_{\pm}(C_n)$. It follows that $\ord(g)-2\in\Delta(\mathsf{L}_{\mathcal{B}(C_n)}(S_k))$, and hence $\ord(g)-2$ is a multiple of $n-3$, a contradiction. Therefore we have $\mathsf{h}(U)\ge 2$.

Next we distinguish two cases depending on $|\supp(U)|$.

\smallskip
\noindent
CASE 1: $|\supp(U)|=2$.

Then there exists $g_1\in\supp(U)$ such that $\mathsf{v}_{g_1}(U)\ge 3$. Let $g_2\in\supp(U)\setminus\{g_1\}$ and set $V=g_1^{-2}(g_1+g_2)(g_1-g_2)U$. If $V$ is not an atom, then $(g_1)^{-2}U$ has a decomposition $(g_1)^{-2}U=T_1T_2$ such that $(g_1+g_2)T_1,(g_1-g_2)T_2\in\mathcal{B}_{\pm}(G)$, whence $g_1g_2T_1,g_1g_2T_2\in\mathcal{B}_{\pm}(G)$ are both subsequences of $U^2$ and $3\le |g_1g_2T_1|,|g_1g_2T_2|<n$. It follows that both $g_1g_2T_1$ and $g_1g_2T_2$ are products of atoms of length $2$, whence $U=g_1^2T_1T_2$ is a product of atoms of length $2$, a contradiction. Therefore $V$ is an atom of length $n$. Similarly we can show that $\mathsf{L}_{\mathcal{B}_{\pm}(G)}(V^2)=\{2,n\}$, a contradiction to the fact that $(g_1+g_2)g_1g_2\t V$.

\smallskip
\noindent
CASE 2: $|\supp(U)|\ge 3$.

Then there exist $g_1\in\supp(U)$ with $\mathsf{v}_{g_1}(U)\ge 2$ and distinct $g_2,g_3\in\supp(U)\setminus\{g_1\}$. Set $V=(g_1g_2)^{-1}(g_1-g_3)(g_2+g_3)U$. Assume to the contrary that $V$ is not an atom. Then $(g_1g_2)^{-1}U$ has a decomposition $(g_1g_2)^{-1}U=T_1T_2$ such that $(g_1-g_3)T_1,(g_2+g_3)T_2\in\mathcal{B}_{\pm}(G)$, whence $g_1g_3T_1,g_2g_3T_2\in\mathcal{B}_{\pm}(G)$ are both subsequences of $U^2$ and $3\le |g_1g_3T_1|,|g_2g_3T_2|<n$. It follows that both $g_1g_3T_1$ and $g_2g_3T_2$ are products of atoms of length $2$, whence $U=g_1g_2T_1T_2$ is a product of atoms of length $2$, a contradiction. Therefore $V$ is an atom of length $n$. Similarly we can show that $\mathsf{L}_{\mathcal{B}_{\pm}(G)}(V^2)=\{2,n\}$, a contradiction to the fact that $(g_1-g_3)g_1g_3\t V$.
\end{proof}

\medskip
\noindent
{\bf Acknowledgement.} We would like to thank the referees for going through the manuscript carefully and for all their comments.

\providecommand{\bysame}{\leavevmode\hbox to3em{\hrulefill}\thinspace}
\providecommand{\MR}{\relax\ifhmode\unskip\space\fi MR }
\providecommand{\MRhref}[2]{\href{http://www.ams.org/mathscinet-getitem?mr=#1}{#2}}
\providecommand{\href}[2]{#2}


\begin{thebibliography}{10}

\bibitem{Ad-Gr-Su12a}
S.D.~Adhikari, D.J.~Grynkiewicz, and Z.-W.~Sun, \emph{On weighted zero-sum sequences}, Adv. Appl. Math. \textbf{48} (2012), 506 -- 527.

\bibitem{Ba-Re22a}
A.~Bashir and A.~Reinhart, \emph{On transfer {K}rull monoids}, Semigroup Forum \textbf{105} (2022), 73 -- 95.

\bibitem{Bi-Ge23}
P.-Y. Bienvenu and A.~Geroldinger, \emph{On algebraic properties of power monoids of numerical monoids}, Israel J. Math., to appear {https://arxiv.org/abs/2205.00982}.

\bibitem{B-M-O-S22}
S.~Boukheche, K.~Merito, O.~Ordaz, and W.A.~Schmid, \emph{Monoids of sequences over finite abelian groups defined via zero-sums with respect to a given set of weights and applications to factorizations of norms of algebraic integers}, Commun. Algebra \textbf{50} (2022), 4195 -- 4217.

\bibitem{Br20a}
C.~Bruce, \emph{{$\rm C^*$}-algebras from actions of congruence monoids on rings of algebraic integers}, Trans. Amer. Math. Soc. \textbf{373} (2020), no.~1, 699 -- 726.

\bibitem{Br21a}
\bysame, \emph{Phase transistions on {$\rm C^*$}-algebras from actions of congruence monoids on rings of algebraic integers}, Int. Math. Res. Not. IMRN (2021), no.~5, 3653 -- 3697.

\bibitem{Cz-Do-Ge16a}
K. Cziszter, M. Domokos, and A. Geroldinger, \emph{The interplay of Invariant Theory with Multiplicative Ideal Theory and with Arithmetic Combinatorics}, in Multiplicative {I}deal {T}heory and {F}actorization {T}heory (2016), Springer, 43 -- 95.

\bibitem{Fa-Zh22a}
V.~Fadinger and Q.~Zhong, \emph{On product-one sequences over subsets of groups}, Periodica Math. Hungarica \textbf{86} (2023), 454 -- 494.

\bibitem{Ge-HK06a}
A.~Geroldinger and F.~Halter-Koch, \emph{Non-{U}nique {F}actorizations. {A}lgebraic, {C}ombinatorial and {A}nalytic {T}heory}, Pure and Applied Mathematics, vol. 278, Chapman \& Hall/CRC, 2006.

\bibitem{Ge-HK-Zh22}
A.~Geroldinger, F.~Halter-Koch, and Q.~Zhong, \emph{On monoids of weighted zero-sum sequences and applications to norm monoids in {G}alois number fields and binary quadratic forms}, Acta Math. Hungarica \textbf{168} (2022), 144 -- 185.

\bibitem{Ge-Ka10a}
A.~Geroldinger and F.~Kainrath, \emph{On the arithmetic of tame monoids with applications to {K}rull monoids and {M}ori domains}, J. Pure Appl. Algebra \textbf{214} (2010), 2199 -- 2218.

\bibitem{Ge-Ka-Re15a}
A.~Geroldinger, F.~Kainrath, and A.~Reinhart, \emph{Arithmetic of seminormal weakly {K}rull monoids and domains}, J. Algebra \textbf{444} (2015), 201 -- 245.

\bibitem{Ge-Oh24a}
A.~Geroldinger and J.S.~Oh, \emph{On the isomorphism problem of monoids of product-one sequences}, {https://arxiv.org/abs/2304.01459}.

\bibitem{Ge-Sc-Zh17b}
A.~Geroldinger, W.A.~Schmid, and Q.~Zhong, \emph{Systems of sets of lengths: transfer {K}rull monoids versus weakly {K}rull monoids}, in Rings, Polynomials, and Modules, Springer, Cham, 2017, pp.~191 -- 235.

\bibitem{Ge-Zh20a}
A.~Geroldinger and Q.~Zhong, \emph{Factorization theory in commutative monoids}, Semigroup Forum \textbf{100} (2020), 22 -- 51.

\bibitem{Gr06a}
D.J.~Grynkiewicz, \emph{A weighted {E}rd{\H{o}}s-{G}inzburg-{Z}iv {T}heorem}, Combinatorica \textbf{26} (2006), 445 -- 453.

\bibitem{Gr13a}
\bysame, \emph{Structural {A}dditive {T}heory}, Developments in Mathematics 30, Springer, Cham, 2013.

\bibitem{Gr-He15a}
D.J.~Grynkiewicz and F.~Hennecart, \emph{A weighted zero-sum problem with quadratic residues}, Uniform Distribution Theory \textbf{10} (2015), 69 -- 105.

\bibitem{Gr-Ma-Or12}
D.J.~Grynkiewicz, L.E.~Marchan, and O.~Ordaz, \emph{A weighted generalization of two conjectures of {G}ao}, Ramanujan J. \textbf{28} (2012), 323 -- 340.

\bibitem{Gr-Ph-Po13}
D.J.~Grynkiewicz, A.~Philipp, and V.~Ponomarenko, \emph{Arithmetic progression weighted subsequence sums}, Israel J. Math. \textbf{193} (2013), 359 -- 398.

\bibitem{HK98}
F.~Halter-Koch, \emph{Ideal {S}ystems. {A}n {I}ntroduction to {M}ultiplicative {I}deal {T}heory}, Marcel Dekker, 1998.

\bibitem{HK14b}
\bysame, \emph{Arithmetical interpretation of {D}avenport constants with weights}, Arch. Math. \textbf{103} (2014), 125 -- 131.

\bibitem{Ka16b}
F.~Kainrath, \emph{Arithmetic of {M}ori domains and monoids{\rm \,:} {T}he {G}lobal {C}ase}, Multiplicative {I}deal {T}heory and {F}actorization {T}heory, Springer Proc. Math. Stat., vol. 170, Springer, 2016, pp.~183 -- 218.

\bibitem{Ma-Or-Ra-Sc13a}
L.E.~Marchan, O.~Ordaz, D.~Ramos, and W.A.~Schmid, \emph{Some exact values of the {H}arborth constant and its plus-minus weighted analogue}, Arch. Math. \textbf{101} (2013), 501 -- 512.

\bibitem{Ma-Or-Ra-Sc16a}
\bysame, \emph{Inverse results for weighted {H}arborth constants}, Int. J. Number Theory \textbf{12} (2016), 1845 -- 1861.

\bibitem{Ma-Or-Sa-Sc15}
L.E.~Marchan, O.~Ordaz, I.~Santos, and W.A.~Schmid, \emph{Multi-wise and constrained fully weighted {D}avenport constants and interactions with coding theory}, J. Comb. Theory, Ser. A. \textbf{135} (2015), 237 -- 267.

\bibitem{Ma-Or-Sc14a}
L.E.~Marchan, O.~Ordaz, and W.A.~Schmid, \emph{Remarks on the plus-minus weighted {D}avenport constant}, Int. J. Number Theory \textbf{10}, 1219 -- 1239.

\bibitem{Ma22}
L.~Margolis, \emph{The modular isomorphism problem: a survey}, Jahresber. Dtsch. Math.-Ver. \textbf{124} (2022), no.~3, 157 -- 196.

\bibitem{Oh20a}
J.S.~Oh, \emph{On the algebraic and arithmetic structure of the monoid of product-one sequences}, J. Commut. Algebra \textbf{12} (2020), 409 -- 433.

\bibitem{Re13a}
A.~Reinhart, \emph{On integral domains that are $\rm{C}$-monoids}, Houston J. Math. \textbf{39} (2013), 1095 -- 1116.

\bibitem{Ro96}
D.~Robinson, \emph{A {C}ourse in the {T}heory of {G}roups, 2nd ed.}, Springer, 1996.

\bibitem{Tr-Ya24a}
S.~Tringali and W.~Yan, \emph{A conjecture by {B}ienvenu and {G}eroldinger on power monoids}, Proc. Amer. Math. Soc., to appear.
\end{thebibliography}
\end{document}